\newif \ifHAL
\HALfalse

\documentclass[leqno]{siamart251104} 
\usepackage[T1]{fontenc}
\usepackage[utf8]{inputenc}
\usepackage[american]{babel}
\usepackage{color}
\usepackage{amssymb}
\usepackage{graphicx}
\usepackage{xspace}
\usepackage{bm,bbm}
\usepackage{centernot}
\usepackage{siunitx}
\usepackage{booktabs}
\newcommand{\mynum}[2]{{\qty[scientific-notation=false, round-mode=figures,round-precision = 5, drop-zero-decimal, round-pad = false]{#1}{#2}}}
\newcommand{\numm}[1]{{\qty[scientific-notation=true, round-mode=figures,round-precision = 3, drop-zero-decimal, round-pad = false]{#1}{}}}

\newcommand{\JLGcode}{\texttt{code1}\@\xspace}
\newcommand{\ETcode}{\texttt{code2}\@\xspace}

\usepackage[numbers,sort]{natbib}
\let\cite=\citet
\usepackage{enumerate,array}
\usepackage{dsfont}
\usepackage{multirow}
\usepackage{enumitem}
\setlist{leftmargin=20pt}
\usepackage{algorithm}
\usepackage{algpseudocodex}
\usepackage{Macros/mydef}

\usepackage{graphics}
\usepackage{wrapfig}
\usepackage{float}
\usepackage{subfig}
\usepackage[percent]{overpic}
\usepackage{varwidth}
\usepackage{tikz}
\usepackage{pgfplots}
\usepackage{adjustbox}
\usetikzlibrary{arrows,matrix,positioning,fit}
\usetikzlibrary{shapes,positioning}
\usetikzlibrary{backgrounds}
\usepackage{tikz-layers}
\usepgfplotslibrary{fillbetween}
\usetikzlibrary{intersections}
\pgfplotsset{compat=1.14}
\usepgfplotslibrary{colorbrewer}
\usepgfplotslibrary{patchplots}
\usepgfplotslibrary[colorbrewer]
\usetikzlibrary{pgfplots.colorbrewer}
\usetikzlibrary[pgfplots.colorbrewer]
\usepgfplotslibrary{units}
\usetikzlibrary{spy}
\usepackage{pgfplotstable}
\usepackage{arrayjobx}
\graphicspath{ {./FIGS/} }
\usetikzlibrary{external}

\colorlet{PlotColor1}{Spectral-A}
\colorlet{PlotColor3}{Spectral-L}
\colorlet{PlotColor2}{Spectral-O}
\colorlet{PlotColor4}{Spectral-D}
\colorlet{PlotColor5}{black}

\begin{document}
\newcommand\footnotemarkfromtitle[1]{%
  \renewcommand{\thefootnote}{\fnsymbol{footnote}}%
  \footnotemark[#1]%
  \renewcommand{\thefootnote}{\arabic{footnote}}}

\title{Invariant-domain preserving IMEX schemes for the nonequilibrium Gray Radiation-Hydrodynamics equations Part I\footnotemark[1]}

\author{Jean-Luc~Guermond\footnotemark[2]
  \and Eric J.~Tovar\footnotemark[3] \footnotemark[4]}

\date{Draft version \today}

\maketitle

\renewcommand{\thefootnote}{\fnsymbol{footnote}}

\footnotetext[2]{Department of Mathematics, Texas A\&M University 3368 TAMU, College
  Station, TX 77843, USA.}%

\footnotetext[3]{Theoretical Division, Los Alamos National Laboratory, P.O. Box 1663, Los Alamos, NM, 87545, USA.}

\footnotetext[4]{Xcimer Energy Corporation, 10325 E 47th Ave, Denver, CO, 80238, USA.}

\footnotetext[1]{This material is based upon work supported in part by
  the National Science Foundation grant DMS2110868, the Air Force
  Office of Scientific Research, USAF, under grant/contract number
  FA9550-23-1-0007, and the U.S. Department of Energy by Lawrence
  Livermore National Laboratory under Contracts B640889. ET
  acknowledges the former support from the U.S. Department of Energy's
  Office of Applied Scientific Computing Research (ASCR) through the
  Competitive Portfolios program at Los Alamos National Laboratory
  (LANL). LANL is operated by Triad National Security, LLC, for the
  National Nuclear Security Administration of the U.S. Department of
  Energy (Contract No. 89233218CNA000001). 
  The LANL release number is LA-UR-26-20761.
  The support of Xcimer Energy Corporation is also acknowledged.}
\renewcommand{\thefootnote}{\arabic{footnote}}
\begin{abstract}
  In this work we introduce an implicit-explicit invariant-domain
  preserving approximation of the nonequilibrium gray
  radiation-hydrodynamics equations.  A time and space approximation
  of the system is proposed using a novel split of the equations
  composed of three elementary subsystems, two hyperbolic and one
  parabolic.  The approximation thus realized is proved to be
  consistent, conservative, invariant-domain preserving, and
  first-order accurate. The proposed method is a stepping stone for
  achieving higher-order accuracy in space and time in the forthcoming
  second part of this work. The method is numerically illustrated and
  shown to converge as advertised. This paper is dedicated to the memory
  of Peter Lax.
\end{abstract}

\begin{keywords}
  Radiation hydrodynamics, nonequilibrium gray diffusion, invariant domain preserving, IMEX, Euler equations
\end{keywords}

\begin{AMS}
  35L65, 65M60, 65M12, 65N30
\end{AMS}

\pagestyle{myheadings} \thispagestyle{plain}
\markboth{J.-L. Guermond and E.~J. Tovar}{IDP IMEX methods for gray radiation hydrodynamics}


\section{Introduction}
The objective of this work is to introduce a first-order approximation
of the nonequilibrium gray radiation hydrodynamics (GRH) equations
that is invariant-domain preserving (IDP), consistent, and
conservative.  The GRH equations play an important role in modeling
the diffusion of thermal radiation in fluids as well as the induced
effects of radiation on the fluid motion. The model is typically used
for applications in optically thick environments (\ie~strong coupling
between radiation and fluid motion) such as inertial confinement
fusion and astrophysics (see:~\cite{baldwin1999iterative}).  The GRH
equations is a system composed of the compressible Euler equations
coupled to a parabolic equation for the radiation energy density where
both sub-systems are supplemented with stiff source terms.  We refer
the reader to~\cite[Sec.~97]{mihalas2013foundations} for a general
overview of the model and~\cite{buet2004asymptotic} where a formal
derivation of the model is presented.

Due to the disparate temporal scales of the radiation and fluid motion, developing 
robust and accurate approximation techniques for the model is challenging.
It is known that applying an explicit time-integration method to the model 
leads to a restrictive time-step, $\Delta t \sim \frac{\sigma_t\times(\Delta x)^2}{c}$, 
where $c$ is the speed of light, $\sigma_t$ is the total absorption opacity, 
and $\Delta x$ is the spatial mesh size (see:~\cite{baldwin1999iterative}).
Alternatively, one could apply implicit time-integration to the full system but this could be 
computationally burdensome. 
A natural approach for overcoming these challenges would be to  
apply operator splitting to the system and use an IMEX time-integration method. 
Some recent approaches in the literature following this idea can be seen
in~\cite{BHEML:17} and \cite{southworth2024implicit}.
An operator split method was applied to the equilibrium diffusion model 
in~\cite{dai1998numerical}.
Another challenge in developing robust numerical methods for the GRH model 
is the stiff non-linearity that arises in the radiation diffusion sub-system 
due to potentially highly contrasted opacities (which depend on density and temperature). 
Efforts in the literature addressing this issue can be seen in~\cite{baldwin1999iterative} 
and~\cite{knoll_rider_Olson_1999} for the radiation diffusion sub-system.
For papers regarding only the radiation diffusion sub-system with a focus on 
time integration techniques, we refer the reader to~\cite{knoll_chacon_margolin_Mousseau_JCP_2003} and~\cite{zheng2024high}. 
For papers regarding positivity of the solution for the radiation diffusion sub-system
and positivity see~\cite{Buet_Despres_JCP_2006} and~\cite{sheng2009monotone}.
Recent work on developing higher-order approximation techniques for the full GRH model can be seen
in~\cite{delchini2015entropy} and~\cite{Delchini_Ragusa_Ferguson_IJNMF_2017}.
Since the nonequilibrium GRH diffusion model is a simplified 
radiation-hydrodynamics model, we refer the reader to~\citet{BHEML:17} and~\cite{He_Wibking_Krumholz_Mark_2024}
for higher-order IMEX schemes applied to generalized radiation-hydrodynamic
models. The literature addressing positivity preservation, or 
more generally, invariant-domain preservation is sparse. 
The purpose of the present paper is to address this issue. 

The contribution of this work is as follows. We introduce a novel
split of the GRH model into two elementary hyperbolic systems and one
parabolic system.  The motivation for introducing two hyperbolic
stages is rooted in the observation that the radiation pressure does
not influence the material internal energy; see \eg Lemma~2.4 in~\cite{Dao_Nazarov_Tomas_JCP_2024}.  
To the authors' best knowledge,
this split seems to be original.  Furthermore,
using~\citep{guermond_popov_sinum_2016}, this split allows for a
straight-forward spatial approximation that is invariant-domain
preserving for each hyperbolic system. We also derive the maximum wave
speed in the local Riemann problem for the second hyperbolic system to
guarantee the IDP property.  Then, we introduce a simple approximation
to the parabolic sub-system using backward Euler time stepping.
This approximation utilizes a fixed-point Picard iteration method and
a local Newton solve for updating the radiation energy density and
material temperature. The fixed-point technique is used to ensure
robustness of the algorithm at high Mach numbers. Finally,
irrespective of the relative tolerance that is used to exit the
fixed-point loop, the approximation is shown to be invariant-domain
preserving under mild assumptions on the equation of state and the
underlying spatial approximation.

The paper is organized as follows. In Section~\ref{sec:the_model} we
recall the full nonequilibrium gray radiation hydrodynamics (GRH) model
and discuss its properties.  We also give a brief background on the
thermodynamics and invariant domain of the system.  In
Section~\ref{sec:novel_split}, we introduce a novel split of the GRH
model which is composed of two hyperbolic subsystems and one parabolic
subsystem. Then, in Section~\ref{Sec:Space_approximation} we discuss
details regarding the spatial approximation and list
some structural assumptions that are invoked later.  Section~\ref{sec:full_scheme} is dedicated to
the approximation technique of the full system.  We focus on the
approximation of the two hyperbolic problems in
Sections~\ref{Sec:approx_stage_1} and~\ref{sec:approx_stage_2}.  We
give a brief discussion on multiplicative \vs additive splitting in
Section~\ref{sec:splitting}.  The main results of these sections are
Lemmas~\ref{lem:IDP_stage_1} and~\ref{lem:IDP_stage_2}.  We discuss
the approximation to the parabolic stage in
Section~\ref{sec:approx_stage_3}.  We detail the fixed-point Picard
iteration and the Newton method in this section.  The main result of
this section is Lemma~\ref{lem:IDP_stage_3} and the main result of the
paper is stated in Theorem~\ref{Thm:low_order_IDP}.  Finally, we conclude
by numerically illustrating the proposed approximation technique.

\section{The model}\label{sec:the_model}
In this section, we introduce the model for nonequilibrium gray radiation hydrodynamics.
We then give a brief discussion on the respective thermodynamics and invariant domain properties.

\subsection{Governing equations}
Let $D$ be a domain in $\polR^d$ where $d = \{1,2,3\}$.  Assume that
the domain is occupied by a radiating fluid that is optically thick.
That is to say, the gradient of the radiation energy density varies
slowly over the photon mean-free path (\cite{bates2001consistent}).
This assumption implies that the fluid motion and radiation are
strongly coupled.  We assume that the opacity of the material is
independent of the frequency of photons in the radiation, \ie the
fluid is a ``gray'' material (\cite{lowrie1999}).  We further assume
that the fluid and radiation fields are not in thermodynamic
equilibrium.  The nonequilibrium gray radiation hydrodynamics
diffusion model corresponding to this situation is written as follows:
\begin{subequations}\label{eq:radiation}\begin{align}
     & \partial_t\rho + \DIV(\bv\rho) = 0,                                      \\
     & \partial_t\bbm + \DIV(\bv\otimes\bbm + p(\bu)\polI_d) +\GRAD p\lorr(\bu)
    = \bzero, \label{mt:eq:radiation}                                           \\
     & \partial_t\Emech + \DIV(\bv(\Emech+p(\bu))) +\bv\SCAL\GRAD p\lorr(\bu)
    = -\sigmaa c(a\lorr \Tmech(\bu)^4-\ER), \label{EM:eq:radiation}             \\
     & \partial_t\ER+ \DIV(\bv\ER) + p\lorr(\bu)\DIV\bv
    - \DIV(\tfrac{c}{3\sigmat}\GRAD\ER) = \sigmaa c(a\lorr \Tmech(\bu)^4-\ER). \label{ER:eq:radiation}
  \end{align}\end{subequations}
Here, the (column) vector of the conserved variables has $d+3$
components $\bu\eqq (\rho,\bbm\tr,\Emech,\ER)\tr$, where $\rho$ is
the density, $\bbm$ is the momentum (viewed as a column vector in
$\Real^d)$, $\Emech$ is the total mechanical energy, and $\ER$ is the
radiation energy (per unit volume).  We define the velocity vector by
$\bv\eqq \rho^{-1}\bbm$.  Here $p\lorr(\bu)$ is the radiation
pressure, $c$ the speed of light, $\sigmaa$ and $\sigmat$ the
absorption and total cross sections (both scale as the inverse of a
length), respectively, $a\lorr\eqq \frac{4\sigma}{c}$ the radiation
constant, $\sigma$ the Stefan--Boltzmann constant, and $p(\bu)$,
$\Tmech(\bu)$ are the mechanical pressure and temperature,
respectively.  More details regarding $p(\bu)$ and $T(\bu)$ are given
in \S\ref{Sec:thermodynamics}. We also define the internal energy
$\varepsilon(\bu) \eqq\Emech -\frac12 \rho\|\bv\|_{\ell^2}^2$, the
specific internal energy $e(\bu)\eqq \rho^{-1}\varepsilon(\bu) \eqq
\rho^{-1}\Emech - \frac12 \|\bv\|_{\ell^2}^2$, and the total energy of
the system by $\Etot\eqq \Emech+\ER$.  For the rest of the paper, we
assume that the radiation pressure is defined by
$p\lorr(\bu)\eqq\frac13\ER$.  We refer the reader to
\S\ref{sec:numerical_prelim} for a discussion on the units used in
this work.

\begin{remark}[Total energy]
  Adding~\eqref{EM:eq:radiation} and~\eqref{ER:eq:radiation}, we
  observe that the total energy of the system satisfies the balance
  equation:
  \begin{equation}
    \partial_t\Etot + \DIV(\bv(\Etot+p(\bu) +p\lorr(\bu))) -
    \DIV(\tfrac{c}{3\sigmat}\GRAD\ER) = 0,
  \end{equation}
  which shows that, in absence of energy source fluxes at the boundary
  of the domain $\Dom$, the total energy is conserved, as expected.
  The three conserved
  variables of the system are the density, $\rho$, the momentum,
  $\bbm$, and the total energy, $\Etot$. The radiation energy is not a
  conserved quantity.
\end{remark}

\begin{remark}[Internal energy]\label{Rem:internal_energy} Taking the
  dot product of \eqref{mt:eq:radiation} with $\bv$ and subtracting
  the result from \eqref{EM:eq:radiation} gives
  the balance equation for the internal energy:
  \begin{equation}
    \partial_t(\rho e(\bu)) + \DIV(\bv\rho e(\bu)) + p(\bu) \DIV \bv=
    -\sigmaa c(a\lorr \Tmech(\bu)^4-\ER),
  \end{equation}
  which shows that the gradient of the radiation pressure has no
  effect whatsoever on the internal energy. We refer the reader to
  Lemma~2.4 in \cite{Dao_Nazarov_Tomas_JCP_2024} where a more general
  statement regarding this property is made. 
\end{remark}

\begin{remark}[Nonconservative products]
  Notice that~\eqref{eq:radiation} has non-conservative
  products $\bv\SCAL\GRAD p\lorr(\bu)$ and $p\lorr(\bu)\DIV\bv$.
  However, due to the presence of the diffusive term $\DIV(\tfrac{c}{3\sigmat}\GRAD\ER)$
  in~\eqref{ER:eq:radiation}, it is reasonable to expect that no discontinuity appears
  in the variable $\ER$ so that these non-conservative
  product are unambiguously defined. The reader is also referred to
  \cite[\S2.4]{Buet_Despres_JCP_2006} where the non-conservative product question is discussed.
\end{remark}

\begin{remark}[Thermodynamic nonequilibrium]
  In the literature, the term \textit{equilibrium} refers to the case
  when the radiation energy density and material temperature satisfy
  the condition $\ER = a\lorr \Tmech(\bu)^4$ for all $\bx\in \Dom$ and
  all $t > 0$.  For a thorough discussion of the equilibrium-diffusion
  limit of the radiation hydrodynamics equations, we refer the reader
  to~\cite{dai1998numerical} and~\cite{ferguson2017equilibrium} and references therein.
\end{remark}

\subsection{Thermodynamics and invariant domain}\label{Sec:thermodynamics}

Throughout the paper, we assume that given a reasonable state $\bu$,
we are able to retrieve the mechanical pressure $p(\bu)$ and
temperature $\Tmech(\bu)$ in a suitable way (\eg by evaluating
arbitrary analytic expressions or by deriving values from tabulated
experimental data).  We call the equation of state that relates these
thermodynamic quantities, the oracle. We assume that we have no a
priori knowledge of the oracle apart from some mild structural
assumptions that we now state. As in
\cite{Clayton_Guermond_Popov_SIAM_SISC_2022}, we assume that the
domain of definition for the thermodynamic quantities is the set
$\calB(b)\subset\Real^{d+3}$ given by
\begin{align}\label{def_of_calB}
  \calB(b)\eqq \big\{\bu \eqq(\rho,\bbm\tr, \Emech,\ER)\tr \in \mathbb{R}^{d+3}
  \st 0<\rho, \ 0 < 1 - b \rho, \; e\lo{cold}(\rho)< e(\bu) \big\}.
\end{align}
The inequality $1 - b \rho > 0$ appearing in the definition of
$\calB(b)$ is the so called maximum compressibility condition. The
constant $b$ can be set to zero if the user has no a priori knowledge
about the maximum compressibility of the fluid under
consideration. The function $e\lo{cold}: \Real_{>0}\to \Real$ is the
cold curve. We henceforth assume that $e\lo{cold}:(0,\frac1b)\to \Real$ is quasiconcave. For the sake of simplicity, we assume that the oracle
returns a pressure and a non-negative temperature:
\begin{align}
  p: \calB(b)\ni \bu & \longmapsto p(\bu)\in \Real,\label{Pressure_is_nonnegative}\\
  \label{def_of_temperature}
  \Tmech: \calB(b)\ni \bu&\longmapsto \Tmech(\bu)\in \Real_{\ge 0}.
\end{align}
The temperature being positive is one of the fundamental principles of
thermodynamics.  Recall that using the definition of the heat capacity
at constant volume $\frak{c}\lo{v}(\rho,T) \eqq \frac{\partial
  e}{\partial T}(\rho,T)>0$, we have $e(\rho,T) = \int_0^T
\frak{c}\lo{v}(\rho,\theta) \diff\theta + e\lo{cold}(\rho)$.  Then, again
for the sake of simplicity, we assume that the oracle gives us access
to the cold curve $e\lo{cold}(\rho)$ and the average heat capacity at
constant volume $\cv(\rho,T) \eqq \frac{1}{T}\int_0^T
\frak{c}\lo{v}(\rho,\theta) \diff\theta$.  We summarize this assumption by
saying that the internal energy and the temperature are related as
follows:
\begin{align}
  e(\bu) = \cv(\rho,\Tmech(\bu)) \Tmech(\bu) + e\lo{cold}(\rho),\qquad \forall \bu\in\calB(b).
  \label{specific_heat_capacity}
\end{align}
We note that the above assumptions can be weakened.  The reader is
referred to \cite{Clayton_Tovar_2025} for more details on how to
weaken these assumptions, but these generalizations are
out of the scope of the paper.

Regarding the radiation quantities, we assume that the absorption and
total cross sections can depend on the oracle; that is to say, we
assume the following properties for  the absorption and
total cross sections
$\sigmaa: \calB(b)\ni \bu \longmapsto \sigmaa(\rho,
  \Tmech(\bu))\in \Real_{\ge0}$ and
$\sigmat: \calB(b)\ni \bu \longmapsto \sigmat(\rho, \Tmech(\bu))\in \Real_{>0}$.

Since it can be shown that the following set
\begin{multline} \label{eq:def_calA_rad}
  \calA(b) \eqq \{ \bu= (\rho,\bbm\tr,\Emech,\ER)\tr \in \Real^{d+3} \st
  0<\rho,\; 0< 1-b\rho,\; \\ e\lo{cold}(\rho)<e(\bu),\; 0<\ER\}
\end{multline}
is invariant under parabolic regularization, we focus our interest
only on those weak solutions of \eqref{eq:radiation} for which
$\calA(b)$ is invariant as well.  Our ultimate objective is to construct an
approximation of~\eqref{eq:radiation} that is high-order accurate in
space and time and leaves $\calA(b)$ invariant. We propose to do so by
first constructing a firs-order method that is IDP, then constructing
a higher-order method that is made IDP by limiting. The objective of
the present paper is to solely focus on the first-order IDP method.

\section{Three elementary subsystems}\label{sec:novel_split}
In this section, we introduce the novel split of the
model~\eqref{eq:radiation}. This split consists of three elementary
subsystems (two hyperbolic and one parabolic). The purpose of this
section is simply to give the reader some heuristics justifying the
decomposition of the system that is used in \S\ref{sec:full_scheme} to perform
the approximation in time.

\subsection{IMEX split (hyperbolic-parabolic)}
The approximation in time of \eqref{eq:radiation} is done by means of an implicit-explicit (IMEX)
time stepping technique.  The explicit time stepping is applied
to the non-stiff part of the problem:
\begin{subequations}\label{Emech_Er:radiation_full}%
  \begin{align}%
     & \partial_t\rho + \DIV(\bv\rho) = 0,                                         \\
     & \partial_t\bbm + \DIV(\bv\otimes\bbm + p(\bu)\polI_d) = -\GRAD p\lorr(\bu), \\
     & \partial_t\Emech + \DIV(\bv(\Emech+p(\bu))) =  -\bv\SCAL\GRAD p\lorr(\bu),  \\
     & \partial_t\ER+ \DIV(\bv \ER) = -p\lorr(\bu)\DIV\bv,
  \end{align}
\end{subequations}
whereas the implicit time stepping is applied of the stiff part:
\begin{subequations}\label{parab}
  \begin{align}
     & \partial_t\rho = 0,                                                                     \\
     & \partial_t\bbm = \bzero,                                                                \\
     & \partial_t\Emech = -\sigmaa c(a\lorr \Tmech(\bu)^4-\ER),                                \\
     & \partial_t\ER- \DIV(\tfrac{c}{3\sigmat}\GRAD\ER) = \sigmaa c(a\lorr \Tmech(\bu)^4-\ER).
  \end{align}
\end{subequations}
We henceforth refer to \eqref{Emech_Er:radiation_full} as the hyperbolic sub-problem (or stage) and we refer to \eqref{parab} as the parabolic sub-problem (or stage).

To be able to construct an approximation method that is consistent,
conservative, and leaves $\calA(b)$ invariant, we further
decompose~\eqref{Emech_Er:radiation_full}.  The decomposition is based
on the observation that the radiation pressure does not have any
influence on the internal energy, as stated in
Remark~\ref{Rem:internal_energy}. We split~\eqref{Emech_Er:radiation_full} into two stages to account for this
fundamental principle.

\subsubsection{Hyperbolic stage 1} In the first stage of the solution
process for \eqref{Emech_Er:radiation_full}, which we henceforth call
hyperbolic stage 1, we just account for the influence of the
mechanical pressure $p(\bu)$ in the time evolution of the momentum and
mechanical total energy. The effect of the radiation pressure is
accounted for in the second stage.  The first stage is formulated as
follows: Given some initial data
$(\rho_0,\bbm_0\tr,{\Emech}_0,{\ER}_0)\tr$, find
$\bu\eqq (\rho,\bbm\tr,\Emech,\ER)\tr$ \sth
\begin{equation}\label{hyperbolic_one}
  \partial_t\bu + \DIV\polg(\bu) = \bm{0},\quad \text{with}\quad
  \polg(\bu)\eqq\begin{pmatrix}\bv\rho                        \\
    \bv\otimes\bbm + p(\bu)\polI_d \\
    \bv(\Emech+p(\bu))             \\
    \bv\ER
  \end{pmatrix}.
\end{equation}
This system is just the compressible Euler equation augmented with an additional
linear conservation equation for the radiation energy. Approximating
this problem in time and space is a standard exercise explained in \S\ref{Sec:approx_stage_1}.

\subsubsection{Hyperbolic stage 2} The second stage, which we
henceforth call hyperbolic stage 2, reads as follows:
Given some initial data $(\rho_0,\bbm_0\tr,{\Emech}_0,{\ER}_0)\tr$,
seek $(\rho,\bbm\tr,\Emech,\ER)\tr$ so that
\begin{subequations}\label{Emech_Er:energy_exchange}\begin{align}
     & \partial_t\rho = 0,                                                           \\
     & \partial_t\bbm =-\GRAD p\lorr(\ER), \label{mt:Emech_Er:energy_exchange}       \\
     & \partial_t\Emech = -\bv\GRAD p\lorr(\ER), \label{Em:Emech_Er:energy_exchange} \\
     & \partial_t\ER =- p\lorr(\ER)\DIV\bv.\label{Er:Emech_Er:energy_exchange}
  \end{align}\end{subequations}
It turns out that the system \eqref{Emech_Er:energy_exchange} can be
further simplified and be put in conservative form. Notice that
\eqref{mt:Emech_Er:energy_exchange} combined with
\eqref{Em:Emech_Er:energy_exchange} implies that the internal energy,
$\varepsilon(\bu) \eqq \rho e(\bu)\eqq \Emech -\frac12 \rho\|\bv\|_{\ell^2}^2$, is constant in
time (in agreement with the statement already made in
Remark~\ref{Rem:internal_energy}, see
\citep[Lem~2.4]{Dao_Nazarov_Tomas_JCP_2024}). Likewise, combining
\eqref{mt:Emech_Er:energy_exchange} and
\eqref{Er:Emech_Er:energy_exchange}, we obtain a conservation equation
for the quantity $E\lo{t}\eqq \ER+\frac12 \rho\|\bv\|_{\ell^2}^2$ (not to be confused with $E\lo{tot}$).  In
conclusion, \eqref{Emech_Er:energy_exchange} can be rewritten in the
following equivalent form:
\begin{subequations}\label{Emech_Er:energy_exchange_reformulated}\begin{align}
     & \partial_t\rho = 0,                                                           \\
     & \partial_t\bbm =-\GRAD p\lorr(\ER), \label{mt:Emech_Er:energy_exchange_reformulated}       \\
     & \partial_tE\lo{t} = -\DIV(\bv p\lorr(\ER)), \label{Et:Emech_Er:energy_exchange_reformulated} \\
     & \partial_t\varepsilon(\bu) =0.\label{epsilon:Emech_Er:energy_exchange_reformulated}
  \end{align}\end{subequations}
Hence, given the initial data
$(\rho_0,\bbm_0\tr,{\Emech}_0,{\ER}_0)\tr$, setting
$E\lo{t0} \eqq {\ER}_0 +\frac12 \rho_0\|\bv_0\|_{\ell^2}^2$, seek
we $\bw\eqq (\rho,\bbm\tr,E\lo{t})\tr$ such that
\begin{equation}\label{hyperbolic_two}
  \partial_t\bw
  +\DIV\polk(\bu) = \bm{0},\quad
  \polk(\bw)\eqq\begin{pmatrix}0 \\
    p_\sfr(\bw)\polI_d \\
    \bv p_\sfr(\bw)
  \end{pmatrix},\qquad
  p_{\sfr}(\bw)\eqq
  \tfrac13( E\lo{t}-\tfrac{1}{2}\rho\|\bv\|_{\ell^2}^2),
\end{equation}
and  $\partial_t\varepsilon(\bu) =0.$
After solving~\eqref{hyperbolic_two} and setting
$\varepsilon_0\eqq {\Emech}_{0} -\frac12 \rho_0\|\bv_0\|_{\ell^2}^2$, the full state
$(\rho,\bbm\tr,{\Emech},{\ER})\tr$ is recovered by setting
\begin{equation}
  \ER\eqq E\lo{t}-\tfrac12 \rho\|\bv\|_{\ell^2}^2,\qquad \Emech \eqq
  \varepsilon_0 + \tfrac12 \rho\|\bv\|_{\ell^2}^2 = {\Emech}_{0}
  -\tfrac12 \rho_0\|\bv_0\|_{\ell^2}^2+ \tfrac12
  \rho\|\bv\|_{\ell^2}^2. \label{hyperbolic_two_update_ER_and_Emech}
\end{equation}

\subsubsection{Parabolic stage}\label{Sec:Parabolic_stage}
Finally, the third stage of the decomposition of the
system~\eqref{eq:radiation} consists of rewriting the parabolic stage
\eqref{parab} as follows: Given some initial data
$(\rho_0,\bbm_0\tr,{\Emech}_0,{\ER}_0)\tr$, set $\bv_0 \eqq \bbm_0 /
\rho_0$, $e_0 \eqq\rho_0^{-1}({\Emech}_0-\frac12
\|\bv_0\|_{\ell^2}^2)$, and $T_0\eqq T(\rho_0,e_0)$, then seek
$(\rho,\bv\tr,T,\ER)\tr$ such that
\begin{subequations}\label{parabolic}
  \begin{align}
     & \partial_t\rho = 0,                                                                                     \\
     & \partial_t\bv = \bzero,                                                                                 \\
     & \rho \partial_t (\cv T)  = -\sigmaa c(a\lorr \Tmech^4-\ER),\label{temp::parabolic}                         \\
     & \partial_t\ER- \DIV(\tfrac{c}{3\sigmat}\GRAD\ER) = \sigmaa c(a\lorr \Tmech^4-\ER), \label{ER:parabolic}
  \end{align}
\end{subequations}
where we used the fact that $\partial_t e=\cv\partial_t T$ in~\eqref{temp::parabolic}.
We recover the internal energy after solving
\eqref{parabolic} by setting $e=\cv T$, which finally give
$\Emech \eqq \rho e + \frac12 \rho\|\bv\|_{\ell^2}^2$.

The three elementary problems~\eqref{hyperbolic_one},~\eqref{hyperbolic_two},
and~\eqref{parabolic} are going to play important roles when we prove that our, yet to be
constructed, low-order method is conservative and invariant-domain
preserving.

\section{Approximation details}\label{Sec:Space_approximation}

Although the numerical tests reported in the paper are done with
continuous finite elements, most of what is said herein is
independent of the spatial discretization. Up to unessential
adaptations, all the theoretical results established below hold for finite
differences, finite volumes, continuous and discontinuous finite
elements.

\begin{assumption}\label{Ass:assumption_on_space_discretization}
  To make the presentation of the method discretization
  agnostic, we make the following assumptions \eqref{ass1}--\eqref{ass5}:

  \manuallabel{ass1}{i} The space approximation of any state functions
  $\bu:\Dom \to \Real^{d+3}$ is entirely defined by a finite
  collection of states $\bsfu\eqq\{\bsfu_i\}_{i\in\calV}$, where the
  coefficients $\bsfu_i$ (henceforth called degrees of freedom) are
  $\Real^{d+3}$-valued, $\calV \eqq \intset{1}{I}$ is the index set
  enumerating the degrees of freedom, and we have set
  $I\eqq\text{card}(\calV)$. We also assume that $\calV$ is
  partitioned into interior degrees of freedom, $\calV\upint$,
  and boundary degrees of freedom $\calV\upbnd$, \ie
  $\calV\eqq \calV\upint\cup\calV\upbnd$ and $\calV\upint\cap\calV\upbnd=\emptyset$.
  For instance, if the
  approximation is done with finite elements using global shape
  function $\{\varphi_i\}_{i\in\calV}$ and
  $\bu_h=\sum_{i\in\calV}\bu_i \varphi_i$ is the approximation of some
  function $\bu$, then the only information that is relevant to us
  regarding the approximate function $\bu_h$ is the collection
  $\{\bu_i\}_{i\in\calV}$. Interior degrees of freedom for Lagrange elements
  are such that $\varphi_{i|\front}=0$ for all $i\in\calV\upint$.

  \manuallabel{ass2}{ii} For every $i\in\calV$, there exists a subset
  $\calV(i)\subsetneq\calV$ that collects the local degrees of freedom
  that interact with $i$, which we call stencil at $i$. We assume that
  $j\in\calV(j)$ iff $i\in \calV(i)$. We denote $\calV^{*}(i)\eqq
    \calV(i){\setminus}\{i\}$.

  \manuallabel{ass3}{iii} The underlying spatial discretization
  provides two $I\CROSS I$ real-valued matrices $\polM\upL$ and
  $\polM\upH$ with the following properties.
  $\polM\upL$ is invertible, diagonal, and is called low-order mass
  matrix. The entries of this matrix are denoted
  $\polM_{ij}\upL = m_i\delta_{ij}$ where $m_i$ is called the mass
  associated with the $i$-th degree of freedom.
  $\polM\upH$ is invertible, symmetric, and is called high-order mass
  matrix. The entries of this matrix are denoted $\polM_{ij}\upH =
    m_{ij}$ and are assumed to be such that $m_{ij}=0$ if
  $j\not\in\calV(i)$, \ie $(\polM\upH\sfX)_i = \sum_{j\in\calV(i)}
    m_{ij}\sfX_j$ for all $\sfX\in\polR^{I}$. The two matrices
  $\polM\upL$ and $\polM\upH$ are used to approximate the identity
  operator.  We assume that
  \begin{equation}
    0< m_i\quad \forall i\in\calV, \qquad  m_i = \sum_{j\in\calV(i)} m_{ji}\quad \forall i\in\calV,\label{ass:on_mass}
  \end{equation}
  to guarantee that $\polM\upL$ and
  $\polM\upH$ carry the same mass. This implies that
  $\sum_{i\in\calV} m_i \sfu_i^n = \sum_{i\in\calV}\sum_{j\in\calV(i)}
    m_{ij}\sfu_j^n$ for all $\sfu\in (\Real^{d+3})^I$, and
  $\sum_{i\in\calV} m_i \bsfu_i^n = \sum_{i\in\calV}\sum_{j\in\calV(i)}
    m_{ij}\bsfu_j^n$ for all $\bsfu\in (\Real^{d+3})^I$. For instance, assuming
  that the approximation is done with continuous finite elements with
  global shape functions $\{\varphi_i\}_{i\in\calV}$, then
  $m_i\eqq\int_\Dom\varphi_i(\bx)\diff x$ and $m_{ij}\eqq\int_\Dom\varphi_i(\bx)\varphi_j(\bx)\diff x$.
  Letting
  $z_h=\sum_{i\in\calV}\sfz_i\varphi_i$ be the approximation of some
  smooth function $z:\Dom \to \Real$, we observe that
  $\int_\Dom \varphi_i(\bx) z(\bx) \diff x\approx \int_\Dom
    \varphi_i(\bx) z_h(\bx) \diff x =
    \sum_{j\in\calV(i)}m_{ij}\sfz_j \approx \sum_{j\in\calV(i)}m_{i}\sfz_i$.

  \manuallabel{ass4}{iv} The underlying spatial discretization provides
  a $I\CROSS I$, $\Real^d$-valued matrix  $\polC$ with
  the following properties. The entries of
  $\polC$ are denoted $\bc_{ij}\in\polR^d$ and are assumed to be such that $\bc_{ij}=0$ if
  $j\not\in\calV(i)$. For all $i\in\calV(i)$, the coefficients
  $\{\bc_{ij}\}_{j\in\calV(i)}$ approximate the gradient operator on
  average in some reasonable sense. We further assume that
  \begin{align}
    \bc_{ij} = -\bc_{ji} \quad \forall (i,j)
    \in \calV\upint\CROSS\calV \cup \calV\CROSS\calV\upint,\quad\qquad
    \sum_{j\in\calV(i)}\bc_{ij} = 0.\label{ass:on_cij}
  \end{align}
  For instance, assuming
  that the approximation is done with continuous finite elements with
  global shape functions $\{\varphi_i\}_{i\in\calV}$, then the
  coefficients $\bc_{ij}\eqq\int_\Dom \varphi_i\GRAD \varphi_j\diff x$
  satisfy this property. Indeed, letting
  $z_h=\sum_{i\in\calV}\sfz_i\varphi_i$ be the approximation of some
  function $z:\Dom\to \Real$, we observe that
  $\int_\Dom \varphi_i(\bx) \GRAD z(\bx) \diff x\approx \int_\Dom
    \varphi_i(\bx) \GRAD z_h(\bx) \diff x =
    \sum_{j\in\calV(i)}\sfz_j\bc_{ij}$, which is the desired property. We
  also observe that $\bc_{ij}=-\bc_{ji}$ if
  $\varphi_i\varphi_{j|\front}=0$.  The partition of unity implies
  property implies $\sum_{j\in\calV(i)}\bc_{ij} = 0$.

  \manuallabel{ass5}{v} The underlying spatial
  discretization provides a matrix $\polK(\rho,T)$ with the following
  properties: $\polK(\rho,T)$ is a real-valued $I\CROSS I$ matrix
  approximating in some sense the diffusion operator
  $\ER\mapsto -\DIV(\frac{c}{3\sigma_t(\rho,T)}\GRAD \ER)$. This matrix
  may depend on some given mass, $\rho$, and temperature distribution,
  $T$, or approximation thereof. The
  entries of $\polK(\rho,T)$ are denoted $\sfk_{ij}(\rho,T)$.
  We finally assume that
  \begin{align}
    \sfk_{ij}=-\sfk_{ji} \ \forall (i,j)\in \calV^2,\quad
    \sfk_{ij}\le 0 \ \forall j\in\calV^*(i),\forall i\in\calV,\quad
    \sum_{j\in\calV(i)}\sfk_{ij} = 0\ \forall i\in\calV.\label{ass:on_kij}
  \end{align}
  For instance,
  with continuous finite elements and global shape functions
  $\{\varphi_i\}_{i\in\calV}$, we have
  $\sfk_{ij}(\rho,T)\eqq \int_\Dom \frac{c}{3\sigma_t(\rho,T)}\GRAD
    \varphi_i\SCAL\GRAD\varphi_j \diff x$.
\end{assumption}

Examples of discretization
techniques satisfying the above assumptions are described
in~\citep{Guermond_Popov_Tomas_2019}.

\begin{definition}[Conservation]\label{def:conservation}%
  We say that a scheme $\{(\varrho_i^n,\bsfm_i^{n\sfT},\sfE^n\lo{m},\sfE^n\lo{r})\}_{i\in\calV}
    \mapsto \{(\varrho_i^{n+1},\bsfm_i^{n+1\sfT},\sfE^{n+1}\lo{m},\sfE^{n+1}\lo{r})\}_{i\in\calV}$ is
  conservative if
  \[\sum_{i\in\calV} m_i \varrho_i^n=\sum_{i\in\calV} m_i \varrho_i^{n+1},\quad
    \sum_{i\in\calV} m_i \bsfm_i^{n}=\sum_{i\in\calV} m_i \bsfm_i^{n+1},\quad
    \sum_{i\in\calV} m_i \sfE_{\textup{tot} i}^n=\sum_{i\in\calV} m_i \sfE_{\textup{tot} i}^{n+1}.
  \]
\end{definition}

\section{First-order IDP scheme}\label{sec:full_scheme}
As our high-order scheme (not presented in this paper) is based on the
combination of an IDP low-order method using forward and backward
Euler time stepping with a high-order IMEX method, we first explain in
this section how to construct the IDP low-order method. The method is
composed of three stages. The first two stages
approximate~\eqref{hyperbolic_one} and~\eqref{hyperbolic_two} using
the forward Euler method, whereas the third stage
solves~\eqref{parabolic} using a linearized version of the backward
Euler method.

\subsection{Hyperbolic stage 1}\label{Sec:approx_stage_1}

Let us assume that the approximation at time $t^n$ of the solution to
\eqref{eq:radiation}, say $\bu_h^n\eqq\sum_{j\in\calV}\bsfu_j^n
  \varphi_i$, is such that $\bsfu_i^n \in\calA(b)$ for all $i\in\calV$
where $\calA(b)$ is defined in \eqref{eq:def_calA_rad}.  Let $\dt$ be
the time step at $t^n$, and let us set $t^{n+1}\eqq t^n + \dt$.

Our first goal is to construct
a low-order IDP update
$\bu_h^{n,1}\eqq\sum_{j\in\calV}\bsfu_i^{n,1} \varphi_i$ of the
solution to the first hyperbolic stage~\eqref{hyperbolic_one}.
We essentially proceed as in
\citep{guermond_popov_sinum_2016} using the
technique from \cite{Clayton_Guermond_Popov_SIAM_SISC_2022} to be able
to use tabulated equations of states.

We first define the following low-order flux for all $i\in\calV$ and all
$j\in\calV(i)$:
\begin{equation} \label{def_low_flux_hyp1}
  \bsfF\up{L,n,1}_{ij} \eqq  -(\polg(\bsfu_j^n)+ \polg(\bsfu_i^n))\bc_{ij}
  + d^{\textup{L},n,1}_{ij}(\bsfu_j^n-\bsfu_i^n),
\end{equation}
where $\bc_{ij}\in\Real^d$ is defined
in~\S\ref{Sec:Space_approximation}, the flux $\polg$ is defined in
\eqref{hyperbolic_one}, and the low-order graph viscosity coefficient
$d^{\textup{L},n,1}_{ij}$ is defined by
\begin{equation}
  d^{\textup{L},n,1}_{ij}\eqq
  \max\big(\wlambda_{\max}(\bn_{ij},\pi^1(\bsfu_i^{n}),\pi^1(\bsfu_j^{n})),
  \wlambda_{\max}(\bn_{ji},\pi^1(\bsfu_j^{n}),\pi^1(\bsfu_i^{n}))\big).\label{def_dij_L1}
\end{equation}
Here, $\wlambda_{\max}(\bn,\pi^1(\bsfu_L),\pi^1(\bsfu_R))$ is any upper bound on the
maximum wave speed in the Riemann problem with
the extended flux $\wpolg(\cdot)\bn$, with $\wpolg(\cdot)$ defined in
\eqref{def_flux_radiation_hyp}, and
\begin{equation}
  \pi^1(\bsfu)\eqq (\varrho, \bbm\SCAL\bn,
  \sfE_m-\tfrac12(\|\bsfv\|_{\ell^2}^2 -(\bsfv\SCAL\bn)^2),\sfE_r)\tr.
  \end{equation}
A source code providing a guaranteed
upper bound $\wlambda_{\max}(\bn,\bu_L,\bu_R)$ for every pressure
oracle satisfying~\eqref{Pressure_is_nonnegative}  is available
at \citep{guermond_jean_luc_2021_4685868}.
We then define the low-order hyperbolic update $\bsfu^{n,1}$ by setting
\begin{equation}\label{low_udate_hyp1}%
  m_i\bsfu_i^{n,1}=m_i\bsfu_i^{n} +\dt \bsfF\up{L,n,1}_{i},
  \qquad \bsfF\up{L,n,1}_{i}\eqq\sum_{j\in\calV(i)}
  \bsfF\up{L,n,1}_{ij}.%
\end{equation}%
\begin{lemma}[$\bu^n\mapsto \bsfu^{n,1}$ is IDP\,\&\,conservative]\label{lem:IDP_stage_1}
  Assume that the space discretization meets the structural
  assumptions \eqref{ass1}--\eqref{ass5} from
  Assumption~\ref{Ass:assumption_on_space_discretization}.  Assume
  that $\bsfu_i^n$ is in $\calA(b)$ for all $i\in\calV$.  Assume that the time step
  satisfies $\dt\le \max_{i\in\calV}$.  Let $\bsfu^{n,1}$
  be defined in \eqref{low_udate_hyp1}. Then
\begin{enumerate}[font=\upshape,label=(\roman*)]
 \item $\bsfu_i^{n,1}$ is in $\calA(b)$ for all $i\in\calV$.
\item The mapping $\bu^n\mapsto \bsfu^{n,1}$ is conservative.
\end{enumerate}
\end{lemma}
\begin{proof}
  See the proof of Theorem~4.6 in
  \citep{Clayton_Guermond_Popov_SIAM_SISC_2022}. The conservation in
  the sense of Definition~\ref{def:conservation} is a consequence of
  the identity $\bc_{ij}=-\bc_{ij}$ which we assume to hold when either $i$ or $j$ is not a
  boundary degree of freedom; see
  Assumption~\ref{Ass:assumption_on_space_discretization}\eqref{ass4}.
\end{proof}

\subsection{Hyperbolic stage 2}\label{sec:approx_stage_2}
We continue with the approximation of the second hyperbolic
stage~\eqref{hyperbolic_two}. For every state
$\bsfu\eqq(\varrho,\bsfm\tr,\sfE_{\sfm},\sfE_{\sfr})\tr$,
we define the reduced state $\bsfw\eqq(\varrho,\bsfm\tr,\sfE\lo{t})\tr$ where
$\sfE\lo{t}\eqq \sfE_{\sfr} +\frac12
  \varrho\|\bsfv\|_{\ell^2}^2$ with $\bsfv\eqq\frac{\bsfm}{\varrho}$.

Given the initial data
$(\varrho^{n},(\bsfm^{n})\tr,\sfE_{\sfm}^{n},\sfE_{\sfr}^{n})\tr$,
we set $\bsfw^{n}\eqq (\varrho^n, (\bsfm^{n})\tr, \sfE\lo{t}^{n})\tr$, with
$\sfE\lo{t}^{n}\eqq \sfE_{\sfr}^{n} +\frac12
  \varrho^{n}\|\bsfv^{n}\|_{\ell^2}^2$. Next we define the low-order
flux corresponding to the nontrivial part of the system of balance
equations~\eqref{hyperbolic_two},
\begin{align} \label{def_low_flux_hyp2}
  \bsfK\up{L,n,2}_{ij} & \eqq  -(\polk(\bsfw_j^{n})+ \polk(\bsfw_i^{n}))\bc_{ij}
  + d^{\textup{L},n,2}_{ij}(\bsfw_j^{n}-\bsfw_i^{n}),                            \\
  \polk(\bsfw)         & \eqq (\bzero, p_{\sfr}(\bsfw)\polI,\bsfv p_{\sfr}(\bsfw))\tr,
  \quad \text{with} \quad p_r(\bsfw)\eqq
  \tfrac13(\sfE\lo{t}-\tfrac12 \varrho\|\bsfv\|_{\ell^2}^2).
\end{align}
The low-order graph viscosity coefficient
$d^{\textup{L},n,2}_{ij}$ is defined for all $i\in\calV$, $j\in\calV^*(i)$, by
\begin{equation}
  d^{\textup{L},n,2}_{ij}\eqq \max(\wmu_{\max}(\bn_{ij},\pi^2_{\bn_{ij}}(\bsfw_i^{n}),
  \pi^2_{\bn_{ij}}(\bsfw_j^{n})),
  \wmu_{\max}(\bn_{ji},\pi^2_{\bn_{ij}}(\bsfw_j^{n}),\pi^2_{\bn_{ij}}(\bsfw_i^{n})),\label{def_dij_L2}
\end{equation}
where $\wmu_{\max}(\bn,\pi^2_{\bn}(\bsfw_L),\pi^2_{\bn}(\bsfw_R))$ is
any upper bound on the maximum wave speed in the Riemann problem
\eqref{app_hyperbolic_two} with $\pi^2_{\bn}(\bsfw)\eqq(\varrho,
\varrho \bsfv\SCAL\bn,\sfE_{\textup{t}}-\frac12
\varrho\|\bsfv\|_{\ell^2}^2+\tfrac12\varrho (\sfv\SCAL\bn)^2 )\tr$.
  Using the definition of $\sfE\lo{t}$, this also gives
  \begin{equation}
    \pi^2_{\bn}(\bsfw) = (\varrho,\varrho\bsfv\SCAL\bn,\sfE_{\textup{r}}
    +\tfrac12\varrho(\bsfv\SCAL\bn)^2)\tr.
  \end{equation}
  All the details regarding
the computation of $\wmu_{\max}$ are given in \S\ref{Sec:Second_Riemann_Problem}.
We define the low-order hyperbolic update
$\bsfw^{n,2}\eqq (\varrho^{n,2},(\sfm^{n,2})\tr,\sfE\lo{t}^{n,2})\tr$ by setting
\begin{equation}\label{low_udate_hyp2}
  m_i\bsfw_i^{n,2}\eqq m_i\bsfw_i^{n} +\dt \bsfK\up{L,n,2}_{i}, \qquad
  \bsfK\up{L,n,2}_{i}\eqq\sum_{j\in\calV(i)}
  \bsfK\up{L,n,2}_{ij}.
\end{equation}
The update
$\bu^{n,2}\eqq (\varrho^{n,2},\bsfm^{n,2},\sfE_{\sfm}^{n,2},\sfE_{\sfr}^{n,2})\tr$ is
then obtained by setting
\begin{subequations}\label{change_var_wn2_to_un2}%
\begin{align}%
\varrho^{n,2} \eqq{}  & \varrho^{n,2},\\
\bsfm^{n,2}   \eqq{}  & \bsfm^{n,2},\\
\sfE_{\sfm}^{n,2}\eqq{}& \sfE_{\sfm}^{n} - \tfrac12\varrho^{n}\|\bsfv^{n}\|_{\ell^2}^2
+\tfrac12\varrho^{n,2}\|\bsfv^{n,2}\|_{\ell^2}^2,\label{change_var_wn2_to_un2:Em}\\
\sfE_{\sfr}^{n,2}\eqq{}& \sfE\lo{t}^{n,2}
-\tfrac12\varrho^{n,2}\|\bsfv^{n,2}\|_{\ell^2}^2.\label{change_var_wn2_to_un2:Er}
\end{align}%
\end{subequations}%
The definitions of the updates
\eqref{change_var_wn2_to_un2:Em}-\eqref{change_var_wn2_to_un2:Er}
follow from~\eqref{hyperbolic_two_update_ER_and_Emech}. That is to
say, the update \eqref{change_var_wn2_to_un2:Er} is a materialization
of the definition $\sfE\lo{t}\eqq
\sfE\lo{r}+\tfrac12\varrho\|\bsfv\|_{\ell^2}^2$, and the update
\eqref{change_var_wn2_to_un2:Em} defines $\sfE_{\sfm}^{n,2}$ by
enforcing the internal energy, $\sfE_{\sfm}-
\tfrac12\varrho\|\bsfv\|_{\ell^2}^2$, to be constant.

\begin{remark}[Density update]Notice that
although the conservation equation for the density is $\partial_t\rho=0$,
the update $\varrho^{n,2}$ is not equal to $\varrho^n$. The actual density update
is given by $\varrho_i^{n,2}=\varrho_i^{n}+\frac{\dt}{m_i}\sum_{j\in\calV(i)}
d^{\textup{L},n,2}_{ij}(\varrho_j^n-\varrho_i^n)$.
\end{remark}
\begin{lemma}[$\bu^{n}\mapsto \bsfu^{n,2}$ is IDP\,\&\,conservative]\label{lem:IDP_stage_2}
  Assume that the structural
  assumptions \eqref{ass1}--\eqref{ass5} from
  Assumption~\ref{Ass:assumption_on_space_discretization} are met.  Assume
  that $\bsfu_i^{n}$ is in $\calA(b)$ for all $i\in\calV$. Assume that the
  time step is chosen so that $\dt\le \max_{i\in\calV} \frac{2}{m_i}
  \sum_{j\in\calV^*(i)}d^{\textup{L},n,2}_{ij}$. Let
  $\bsfu^{n,2}$ be defined in~\eqref{change_var_wn2_to_un2} with
  $\bsfw^{n,2}$ defined in~\eqref{low_udate_hyp2}. Then
  \begin{enumerate}[font=\upshape,label=(\roman*)]
  \item $\bsfu_i^{n,2}$ is in $\calA(b)$ for all $i\in\calV$.
    \item The mapping
   $\bu^{n}\mapsto \bsfu^{n,2}$ is conservative.
   \end{enumerate}
\end{lemma}
\begin{proof}
  We apply the generic theory developed in
  \citep{guermond_popov_sinum_2016}; in particular, we invoke
  Theorem~4.1 therein. We start be defining
  \[
    \overline{\bsfw}_{ij}^n \eqq \tfrac12(\bsfw_i^{n}+\bsfw_j^{n}) -
    \tfrac12 (\polk(\bsfw_j^{n})
    -\polk(\bsfw_i^{n}))\SCAL\bn_{ij}\tfrac{\|\bc_{ij}\|_{\ell^2}}{d^{\textup{L},n,2}_{ij}}.
  \]
  After rearranging the terms in \eqref{low_udate_hyp2} and using that
  $\sum_{j\in\calV(i)}\bc_{ij}=0$, we obtain
  \[
    \bsfw_i^{n,2} = \bsfw_i^{n}\Big(1 - \frac{2\dt}{m_i}\sum_{j\in\calV^*(i)}d^{\textup{L},n,2}_{ij}\Big)
    + \sum_{j\in\calV^*(i)}\frac{2}{m_i}d^{\textup{L},n,2}_{ij} \overline{\bsfw}_{ij}^n.
  \]
  Thanks to the assumption we made on the time step, the above identity
  is a convex combination. Thanks to the definition of
  $\wmu_{\max}(\bn_{ij},\pi^2_{ij}(\bsfu_i^{n}),\pi^2_{ij}(\bsfu_j^{n}))$ and
  $d^{\textup{L},n,2}_{ij}$ it can be shown
  that $\overline{\bsfw}_{ij}^n$ is a space average of the exact
  solution to the Riemann problem with flux $\polk(\bv)\bn_{ij}$ and
  with left state $\bsfw_i^n$ and
  right state $\bsfw_j^n$.  Let
  us consider the domain
  \[
  \calR\eqq \{\bsfw\eqq (\rho,\bbm\tr,E\lo{t})\in
  \Real^{d+2}\st \rho>0, 1-b\rho>0, E\lo{t}-\tfrac12 \rho\|\bv\|_{\ell^2}^2 > 0\}.
  \]
  Since $\bsfu_i^{n}\in\calA(b)$ and $\bsfu_j^{n}\in\calA(b)$, we
  conclude that $\bsfw_i^{n}\in\calR$ and $\bsfw_j^{n}\in\calR$. As
  the domain $\calR$ is invariant under the action of the (entropy) solution operator
  of the Riemann problem and is convex, we conclude using Jensen's inequality
  that space averages of the exact solution to the Riemann problem
  remain in $\calR$. This in turn implies that $\overline{\bsfw}_{ij}^n$
  is in $\calR$.  Invoking again the convexity of $\calR$, we conclude
  that $\bsfw_i^{n,2}$ is in $\calR$ because $\bsfw_i^{n,2}$ is a convex
  combination of states in $\calR$; hence, the radiation energy of the
  state $\bsfu_i^{n,2}$ defined in \eqref{change_var_wn2_to_un2:Er} is
  positive.  The internal energy of the state $\bsfu_i^{n,2}$ defined in
  \eqref{change_var_wn2_to_un2:Em} is above the cold curve because
  $\varrho_i^{n,2}e(\bsfu_i^{n,2})\eqq
    \sfE_{\sfm}^{n,2}-\tfrac12\varrho^{n,2}\|\bsfv^{n,2}\|_{\ell^2}^2 \eqq
    \sfE_{\sfm}^{n} - \tfrac12\varrho^{n}
    \|\bsfv^{n}\|_{\ell^2}^2\qqe \varrho_i^{n}e(\bsfu_i^{n}) >e\lo{cold}(\varrho_i^{n})$. Likewise we have
  $0< \varrho_i^{n,2}$, $0<1-b\varrho_i^{n,2}$. In conclusion $\bsfu_i^{n,2}$ is
  in $\calA(b)$ for all $i\in\calV$.

  The conservation of mass and momentum in 
  the sense of Definition~\ref{def:conservation} is a consequence of
  the identity $\bc_{ij}=-\bc_{ij}$ which we assume to hold when either $i$ or $j$ is not a
  boundary degree of freedom; see
  Assumption~\ref{Ass:assumption_on_space_discretization}\eqref{ass4}.
  Let us now verify that the total energy is conserved.
  Adding \eqref{change_var_wn2_to_un2:Em} and \eqref{change_var_wn2_to_un2:Er} we obtain
  \[\sfE_{\sfr i}^{n,2}+\sfE_{\sfm i}^{n,2}  =  \sfE_{\textup{t} i}^{n,2}
    + \sfE_{\sfm i}^{n} - \tfrac12\varrho_i^{n}
    \|\bsfv_i^{n}\|_{\ell^2}^2 .\]
  Summing over $i\in\calV$ gives
  \[
    \sum_{i\in\calV } m_i\sfE_{\textup{tot} i}^{n,2}  =   \sum_{i\in\calV } m_i \sfE_{\textup{t} i}^{n,2}
    + \sum_{i\in\calV } m_i (\sfE_{\sfm i}^{n} - \tfrac12\varrho_i^{n}
    \|\bsfv_i^{n}\|_{\ell^2}^2) .
  \]
  But $ \sum_{in\in\calV } m_i \sfE_{\textup{t} i}^{n,2} =
    \sum_{in\in\calV } m_i \sfE_{\textup{t} i}^{n}$ because we
  assumed that $\bc_{ij}=-\bc_{ij}$ when either $i$ or $j$ is not a
  boundary degree of freedom; see
  Assumption~\ref{Ass:assumption_on_space_discretization}\eqref{ass4}.
  The definition of $\sfE_{\textup{t} i}^{n}$ gives
  $\sum_{i\in\calV } m_i \sfE_{\textup{t} i}^{n}= \sum_{i\in\calV
    } m_i (\sfE_{\sfr i}^{n}+ \tfrac12\varrho_i^{n}
    \|\bsfv_i^{n}\|_{\ell^2}^2)$. Hence
  \[
    \sum_{i\in\calV } m_i\sfE_{\textup{tot} i}^{n,2}  =
    \sum_{i\in\calV } m_i (\sfE_{\sfr i}^{n,2} + \sfE_{\sfm i}^{n} ) = \sum_{i\in\calV } m_i\sfE_{\textup{tot} i}^{n}.
  \]
  Hence, the total energy is conserved. This proves that the scheme is
  conservative. This completes the proof.
\end{proof}

We now give some details on how to implement the second hyperbolic
stage using the dependent variables
$(\varrho,\bsfm\tr,\sfE_m,\sfE_r)\tr$ instead of using
$(\varrho,\bsfm\tr,\sfE\lo{t})$, \eqref{low_udate_hyp2} and
\eqref{change_var_wn2_to_un2}. Recalling that $\sfE_{\sft}\eqq \sfE_{\sfr}+\frac12
\varrho\|\bsfv\|_{\ell_2}^2$, we define $\bsfF_{ij}\up{L,n,2}\eqq
(\sfF_{\varrho ij}\up{L,2},(\bsfF_{\bsfm ij}\up{L,2})\tr,
\sfF_{\sfE_\sfm ij}\up{L,2}, \sfF_{\sfE_\sfr ij}\up{L,2})\tr$ and
$\bsfB_{i}\up{L,n,2}\eqq (\sfB_{\varrho i}\up{L,2},(\bsfB_{\bsfm
  i}\up{L,2})\tr, \sfB_{\sfE_\sfm i}\up{L,2},\sfB_{\sfE_\sfr
  i}\up{L,2})\tr$ where
\begin{subequations}\begin{align}
     & \left\{\begin{aligned}
\sfF_{\varrho ij}\up{L,n,2} ={}   & d^{\textup{L},n,2}_{ij}(\varrho_j^{n}-\varrho_i^{n}),   \\
\bsfF_{\bsfm ij}\up{L,n,2}   ={}  & \textstyle
-\bc_{ij}(p_{\sfr}(\sfE_{\sfr,j}^{n}) +p_{\sfr}(\sfE_{\sfr,i}^{n})) + d^{\textup{L},n,2}_{ij}(\bsfm_j^{n}-\bsfm_i^{n}), \\
\sfF_{\sfE_\sfm ij}\up{L,n,2} ={}  & 0,                   \\
\sfF_{\sfE_\sfr ij}\up{L,n,2}  ={} & -\bc_{ij}\SCAL \left(\bsfv_j^{n}
p_{\sfr}(\sfE_{\sfr,j}^{n}) + \bsfv_i^{n} p_{\sfr}(\sfE_{\sfr,i}^{n})\right)+ d^{\textup{L},n,2}_{ij}(\sfE_{\sft,j}^{n}-\sfE_{\sft,i}^{n})\end{aligned}\right.\\
     & \left\{\begin{aligned}
\bsfB_{\varrho i}\up{L,n,2} ={} & 0,  \\
\bsfB_{\bsfm i}\up{L,n,2} ={}   & 0,  \\
\bsfB_{\sfE_\sfm i}\up{L,n,2}={} & m_i\tfrac{\varrho_i^{n,2}}{2}(\|\bsfv_i^{n,2}\|_{\ell_2}^2-\|\bsfv_i^{n}\|_{\ell_2}^2),\\
\bsfB_{\sfE_\sfr i}\up{L,n,2}={} & -m_i\tfrac{\varrho_i^{n,2}}{2}(\|\bsfv_i^{n,2}\|_{\ell_2}^2-\|\bsfv_i^{n}\|_{\ell_2}^2).
\end{aligned}\right. \label{def_of_B2L}
  \end{align}\end{subequations}
Then the update~\eqref{change_var_wn2_to_un2} can be rewritten into the following equivalent form:
\begin{equation} \label{low_udate_hyp2_alternative}
  m_i\bsfu_i^{n,2}=m_i\bsfu_i^{n} +\dt \bsfF\up{L,n,2}_{i}+ \bsfB\up{L,n,2}_{i} ,
  \qquad \bsfF\up{L,n,2}_{i}\eqq\sum_{j\in\calV(i)}
  \bsfF\up{L,n,2}_{ij}.
\end{equation}
Notice that in \eqref{low_udate_hyp2_alternative} the velocity $\bsfv_i^{n,2}$ has to be updated before
updating $\sfE_{\sfm,i}^{n,2}$ and $\sfE_{\sfr,i}^{n,2}$ because the sources
$\bsfB_{\sfE_\sfm i}\up{L,n,2}$ and $\bsfB_{\sfE_\sfr i}\up{L,n,2}$ depend on $\bsfv_i^{n,2}$; see
\eqref{def_of_B2L}.

\subsection{Multiplicative vs. additive splitting}\label{sec:splitting}
The hyperbolic update, $\bu^{\textup{h},n}$, can be realized in two
different ways: either multiplicative or additive. We now discuss
these two options.  It turns out that the additive update is the most
robust method.

\subsubsection{Multiplicative splitting} The multiplicative version of the hyperbolic update,
$\bu^{\textup{h},n}$, consists of handling the hyperbolic stages 1 and
2 in a sequential way, where stage 1 is followed by stage 2 (or vice
versa).  This multiplicative process can be symbolically represented
by $\bu^n\mapsto \bu^{n,1}\mapsto \bu^{n,2}\qqe \bu^{\textup{h},n}$
where the time step $\dt$ for this explicit algorithm depends on the
state $\bu^n$.  The technical difficulty with this process is that the
time step $\dt$ for steps 1 and 2 must be identical and the second
hyperbolic step can be guaranteed to be IDP only if $\dt\le
\max_{i\in\calV}\frac{2}{m_i}\sum_{j\in\calV^*(i)}d^{\textup{L},n,2}_{ij}$
where $d^{\textup{L},n,2}_{ij}$ is defined in \eqref{def_dij_L2};
hence, the time step $\dt$ a priori depends on the result of stage 1.
This means that the time step is implicitly defined. Hence, the
multiplicative splitting does not have a guaranteed way to choose $\dt$
so the mapping $\bu^n\mapsto \bu^{n,1}\mapsto \bu^{n,2}\qqe
\bu^{\textup{h},n}$ is invariant-domain preserving.

\subsubsection{Additive splitting}
A better way to proceed, often advocated in the literature, is to make
the splitting additive. Given the state $\bu^n$, we define
$\bsfw^{n}\eqq (\varrho^n,(\bsfm^{n})\tr, \sfE\lo{t}^{n})\tr$, with
$\sfE\lo{t}^{n}\eqq \sfE_{\sfr}^{n} +\frac12
\varrho^{n}\|\bsfv^{n}\|_{\ell^2}^2$. Then for all $i\in\calV$ and all
$j\in\calV^*(i)$, we define
\begin{subequations}\label{definition_of_dt1_and_dt2}\begin{align}
  d^{\textup{L},n,1}_{ij}&\eqq \max\big(\wlambda_{\max}(\bn_{ij},\pi^1_{ij}(\bsfu_i^{n}),\pi^1_{ij}(\bsfu_j^{n})),
  \wlambda_{\max}(\bn_{ji},\pi^1_{ij}(\bsfu_j^{n}),\pi^1_{ij}(\bsfu_i^{n}))\big).\\
  d^{\textup{L},n,2}_{ij}&\eqq \max\big(\wmu_{\max}(\bn_{ij},\pi^2_{ij}(\bsfw_i^{n}),\pi^2_{ij}(\bsfw_j^{n})),
  \wmu_{\max}(\bn_{ji},\pi^2_{ij}(\bsfw_j^{n}),\pi^2_{ij}(\bsfw_i^{n}))\big),
\end{align}\end{subequations}
and introduce the two time steps
\begin{equation}
  \dt_1 \eqq  \max_{i\in\calV}\frac{2}{m_i}\sum_{j\in\calV^*(i)}d^{\textup{L},n,1}_{ij},\qquad
  \dt_2 \eqq  \max_{i\in\calV}\frac{2}{m_i}\sum_{j\in\calV^*(i)}d^{\textup{L},n,2}_{ij}.
\end{equation}
Notice that depending on the Mach number, the two time steps $\dt_1$ and $\dt_2$ may be
significantly different. We then define
\begin{equation}\label{def_of_dt_and_theta}
  \theta \eqq \frac{\dt_2}{\dt_1+\dt_2},\qquad\dt \eqq
\frac{\dt_1\dt_2}{\dt_1+\dt_2}.
\end{equation}
Notice that these definitions imply that $ \dt = \theta \dt_1$, $\dt = (1-\theta)\dt_2$, and
$\dt<\min(\dt_1,\dt_2)$. Using the notation
described in \eqref{low_udate_hyp1} and \eqref{low_udate_hyp2}, we define
the additive updates
\begin{subequations}\begin{align}
 \bu^{n,1} &=\bsfu^n +\frac{\dt}{m_i\theta}\bsfF^{\textup{L},n,1},\\
 \bu^{n,2} &=\bsfu^n +\frac{\dt}{m_i(1-\theta)}\bsfF^{\textup{L},n,2} + \bsfB^{\textup{L},n,2}_i,
 \end{align}
\end{subequations}
where $\bsfF^{\textup{L},n,1}$, $\bsfF^{\textup{L},n,2}$, and
$\bsfB^{\textup{L},n,2}$ are constructed using the same state initial $\bu^n$.
Notice that $\bu^{n,1}$ is the first hyperbolic update realized with the time step $\dt_1$,
and $\bu^{n,2}$ is the second hyperbolic update realized with the time step $\dt_2$.
The final hyperbolic update $\bu^{\textup{h},n}$ is defined by
\begin{equation}
  \bu^{\textup{h},n} \eqq \theta \bu^{n,1} + (1-\theta) \bu^{n,2}. \label{final_hyperbolic_update}
  \end{equation}

\begin{theorem}[$\bu^{n}\mapsto \bsfu^{h,n}$ is IDP\,\&\,conservative]%
\label{thm:hyperbolic_stage_is_IDP_and_conservative}  Assume that the space discretization meets the structural
  assumptions \eqref{ass1}--\eqref{ass5} from
  Assumption~\ref{Ass:assumption_on_space_discretization}.  Assume
  that $\bsfu_i^{n}$ is in $\calA(b)$ for all $i\in\calV$.  Assume
  that the time step is chosen so that $\dt\le
  \frac{\dt_1\dt_2}{\dt_1+\dt_2}$ with $\dt_1$ and $\dt_2$ defined in
  \eqref{definition_of_dt1_and_dt2}. Let
  $\bsfu^{h,n}$ be defined in~\eqref{final_hyperbolic_update}.  Then
\begin{enumerate}[font=\upshape,label=(\roman*)]
  \item $\bsfu_i^{\textup{h},n}$ is in
    $\calA(b)$ for all $i\in\calV$.
 \item The mapping
    $\bu^{n}\mapsto \bsfu^{\textup{h},n}$ is conservative.
    \end{enumerate}
\end{theorem}
\begin{proof}
  Notice that $\bu^n\mapsto \bu^{n,1}$ is conservative and the mapping
  is IDP because $\frac{\dt}{\theta} = \dt
  \frac{\dt_1+\dt_2}{\dt_2}= \dt_1$.  The same argument holds for
  $\bu^n\mapsto \bu^{n,2}$ because $\frac{\dt}{1-\theta} = \dt
  \frac{\dt_1+\dt_2}{\dt_1}= \dt_2$.  Observe finally that
  $\bsfu^{\textup{h},n}$ is a convex combination of two states that
  are IDP and conservative.
\end{proof}

All the tests reported in the paper are done with the additive update.

\subsection{Parabolic stage}\label{sec:approx_stage_3}

We finally focus our attention on the solution to the parabolic stage
\eqref{parabolic}.  Given the state
$\bu^{\textup{h},n}\eqq(\varrho^{\textup{h},n},\bsfm^{\textup{h},n},\sfE_{\sfm}^{\textup{h},n},\sfE_{\sfr}^{\textup{h},n})\tr$,
we define the specific internal energy $\sfe^{\textup{h},n}\eqq
\frac{1}{\varrho^{\textup{h},n}}
(\sfE_{\sfm}^{\textup{h},n}-\frac12\varrho^{\textup{h},n}\|\bsfv^{\textup{h},n}\|_{\ell^2}^2)$.
Recalling the relation between the temperature and the internal energy
\eqref{def_of_temperature} given by the oracle, we set
$\sfT^{\textup{h},n}\eqq
T(\varrho^{\textup{h},n},\sfe^{\textup{h},n})$. To account for the
dependency of the other coefficients with respect to the temperature,
we use a standard linearization process; see \eg
\cite[Eqs.~(18)-(19)]{knoll_rider_Olson_1999},
\cite[Eqs.~(42)-(43)]{knoll_chacon_margolin_Mousseau_JCP_2003}.  We
denote by $\sfT^*$ a positive and yet to be computed estimation of
$\sfT^{n+1}$, and we set $\sigmaa^{\textup{h},n}(\sfT^*)\eqq
\sigmaa(\varrho^{\textup{h},n},\sfT^{*})$,
$\cv^{\textup{h},n}(\sfT^*)\eqq \cv(\varrho^{\textup{h},n},\sfT^{*})$
where the average heat capacity at constant volume $\cv$ is define in
\eqref{specific_heat_capacity}.  The coefficients of the matrix
$\polK(\rho^{\textup{h},n},T^{*})$ introduced
in~Assumption~\ref{Ass:assumption_on_space_discretization}\eqref{ass5}
are denoted $\sfk_{ik}^{\textup{h},n}(\sfT^*)$.

 Let $\bsfu^{n+1}\eqq
(\varrho^{n+1},\bsfm^{n+1},\sfE_{\sfm}^{n+1},\sfE_{\sfr}^{n+1})\tr$
 be the low-order parabolic update. Given $\sfT^*>0$ (yet to be clearly defined),
 we update $(\sfT^{n+1})_{i\in\calV}$ and $(\sfE_r^{n+1})_{i\in\calV}$
 by solving the following discrete counterpart
of \eqref{parabolic}:
\begin{subequations}\label{parabolic_update}
  \begin{align}
    \varrho_i^{\textup{h},n}(\cv^{\textup{h},n}(\sfT_i^*) \sfT_i^{n+1}-\cv^{\textup{h},n}(\sfT_i^{\textup{h},n})\sfT_i^{\textup{h},n})
    = - \dt \sigmaa^{\textup{h},n}(\sfT^*)_i c\, (a\lorr[\sfT_i^{*}]^3\sfT_i^{n+1} -
    \sfE_{\sfr,i}^{n+1}),&\label{EM:parabolic_update}                                                                          \\
    m_i(\sfE_{\sfr,i}^{n+1}-\sfE_{\sfr,i}^{\textup{h},n})+\dt(\polK(\rho^{\textup{h},n},T^*)\sfE^{n+1})_i
 = \dt m_i\sigmaa^{\textup{h},n}(\sfT^*_i) c\, (a\lorr [\sfT_i^{*}]^3\sfT_i^{n+1} -
    \sfE_{\sfr,i}^{n+1}).\hspace{-.5cm}& \label{ER:parabolic_update}
  \end{align}\end{subequations}
Finally, recalling \eqref{specific_heat_capacity} we set $\sfe^{n+1} \eqq \cv(\varrho^{\textup{h},n},\sfT^*)\sfT^{n+1}
+e\lo{cold}(\varrho^{\textup{h},n})$.
The mechanical energy 
and the other components of the parabolic update are obtained by setting%
\begin{align}\label{end_of_parabolic_update}
  \varrho^{n+1}\eqq \varrho^{\textup{h},n},\quad
\bsfm^{n+1}\eqq \bsfm^{\textup{h},n},\quad
\sfE_{\sfm}^{n+1}&\eqq \varrho^{n+1} (\sfe^{n+1} + \tfrac12
\|\bsfv^{n+1}\|_{\ell^2}^2),\\
\text{with}\quad \sfe^{n+1} &\eqq \cv(\varrho^{n+1},\sfT^*)\sfT^{n+1}
+e\lo{cold}(\varrho^{\textup{h},n}).\label{update_of_the_internal_energy}
\end{align}

\begin{lemma}[$\bu^{\textup{h},n}\mapsto \bsfu^{n+1}$ is IDP\,\&\,conservative]\label{lem:IDP_stage_3}
  Assume that the space discretization meets the structural
  assumptions \eqref{ass1}--\eqref{ass5} from
  Assumption~\ref{Ass:assumption_on_space_discretization}. Assume that
  $\sfT^*_i>0$ for all $i\in\calV$.  Then
\begin{enumerate}[font=\upshape,label=(\roman*)]
  \item 
  The system
  \eqref{parabolic_update} is linear and has a unique solution.
  \item The
  low-order parabolic update $\bsfu^{\textup{h},n}\mapsto \bsfu^{n+1}$ is IDP
  for all $\dt>0$.
\item The scheme $\bsfu^{\textup{h},n}\mapsto \bsfu^{n+1}$ is conservative.
  \end{enumerate}
\end{lemma}
\begin{proof}
  Re-arranging the terms in \eqref{EM:parabolic_update} yields
  \begin{equation}
    \sfT_i^{n+1} = \frac{\varrho_i^{\textup{h},n}\cv^{\textup{h},n}(\sfT_i^{\textup{h},n}) \sfT_i^{\textup{h},n}
      +  \dt\sigmaa^{\textup{h},n}(\sfT^*_i) c\, \sfE_{\sfr,i}^{n+1}}{\varrho_i^{\textup{h},n}\cv^{\textup{h},n}(\sfT_i^*)
      + \dt \sigmaa^{\textup{h},n}(\sfT^*_i) c\, a\lorr\,[\sfT_i^{*}]^3}, \label{T:prof:lem:IDP_stage_3}
  \end{equation}
  which proves that the dependency $\sfE_{\sfr}^{n+1}\mapsto \sfT^{n+1}$
  is affine. Hence, the system
  \eqref{parabolic_update} is linear.
  The identity \eqref{T:prof:lem:IDP_stage_3} proves that the new
  temperature $\sfT_i^{n+1}$ is positive once we establish that
  $\sfE_{\sfr}^{n+1}\ge 0$, which we now prove.  Substituting
  \eqref{T:prof:lem:IDP_stage_3} into \eqref{ER:parabolic_update}
  gives
  \begin{multline}
    m_i\bigg(1 + \frac{\dt\sigmaa^{\textup{h},n}(\sfT^*_i) c
      \varrho_i^{\textup{h},n}\cv^{\textup{h},n}(\sfT^*_i)}{\varrho_i^{\textup{h},n}\cv^{\textup{h},n} + \dt
      \sigmaa^{\textup{h},n}(\sfT^*_i) c a\lorr[\sfT_i^{*}]^3}\bigg)
    \sfE_{\sfr,i}^{n+1}+ \dt\sum_{j\in\calV(i)}
    \sfk_{ij}^{\textup{h},n}(\sfT^*)\sfE_{\sfr,j}^{n+1} = \\ m_i
    \underbrace{\bigg(\sfE_{\sfr,i}^{\textup{h},n} +
      \frac{\dt\sigmaa^{\textup{h},n}(\sfT^*_i)c\,\varrho_i^{\textup{h},n} \cv^{\textup{h},n}(\sfT_i^{\textup{h},n})
      }{\varrho_i^{\textup{h},n}\cv^{\textup{h},n}(\sfT_i^{\textup{h},n}) + \dt \sigmaa^{\textup{h},n}(\sfT^*_i) c
        a\lorr[\sfT_i^{*}]^3} a\lorr
           [\sfT_i^{*}]^3\sfT_i^{\textup{h},n}\bigg)}_{>0}. \label{ER:proof:lem:IDP_stage_3}
  \end{multline}
  Let $\polG$ be the $I\CROSS I$ matrix with entries $\polG_{ij} \eqq
  m_i(1+\frac{\dt m_i \sigmaa^{\textup{h},n}(\sfT^*_i) c
    \varrho_i^{\textup{h},n}\cv^{\textup{h},n}(\sfT_i^*)}{\varrho_i^{\textup{h},n}\cv^{\textup{h},n}(\sfT_i^{\textup{h},n}) + \dt
    \sigmaa^{\textup{h},n}(\sfT^*_i) c a\lorr[\sfT_i^{*}]^3}) \delta_{ij}
  +\dt \sfk_{ij}^{\textup{h},n}$. Since $\sfk_{ij}^{\textup{h},n}\le 0$ for all
  $j\in\calV^*(i)$, we conclude that $\polG$ is a $Z$-matrix; see
  \eqref{ass:on_kij}.  Since $\sum_{j\in\calV(i)} \sfk_{ij}^{\textup{h},n}=0$,
  we conclude that $\polG$ is an $M$-matrix; see \eg
  \citep[Lem.~28.17]{ern_guermond_volII_2021}.  Hence, $\polG$ is
  invertible and the system
  \eqref{parabolic_update} has a unique solution.

  Since the inverse of $\polG$ has nonnegative entries and the
  right-hand side in \eqref{ER:proof:lem:IDP_stage_3} is positive, we
  infer that $\sfE_{\sfr,i}^{n+1}>0$ for all $i\in\calV$.  More
  precisely, we have
  \[
    \min_{i\in\calV} \sfE_{\sfr,i}^{n+1} \ge\min_{i\in\calV} \tfrac{\sfE_{\sfr,i}^{\textup{h},n} +
      \frac{\dt\sigmaa^{\textup{h},n}(\sfT^*_i)c\varrho_i^{\textup{h},n} \cv^{\textup{h},n}(\sfT_i^{\textup{h},n}) }{\varrho_i^{\textup{h},n}\cv^{\textup{h},n}(\sfT_i^{\textup{h},n})
        + \dt \sigmaa^{\textup{h},n}(\sfT^*_i) c a\lorr[\sfT_i^{*}]^3} a\lorr [\sfT^*_i]^3 \sfT_i^{\textup{h},n}}{1
      + \frac{\dt\sigmaa^{\textup{h},n}(\sfT^*_i) c \varrho_i^{\textup{h},n}\cv^{\textup{h},n}(\sfT_i^*)}{\varrho_i^{\textup{h},n}\cv^{\textup{h},n}(\sfT_i^{\textup{h},n})
        + \dt \sigmaa^{\textup{h},n}(\sfT^*_i) c a\lorr[\sfT_i^{*}]^3}}\ge
    \min_{i\in\calV} \big(\sfE_{\sfr,i}^{\textup{h},n}, a\lorr [\sfT_i^{*}]^3\sfT^{\textup{h},n}_i\big).
  \]
  Then \eqref{T:prof:lem:IDP_stage_3} implies that $\sfT_{i}^{n+1}>0$
  for all $i\in\calV$. Using the relation
  \eqref{update_of_the_internal_energy}, \ie $\sfe_i^{n+1}
  =\cv^{\textup{h},n}(\varrho_i^{n+1}, \sfT_i^{n+1}) +
  e\lo{cold}(\varrho_i^{n+1})$, we infer that the internal
  energy is above the cold curve.  In conclusion, the new state
  $\bsfu_i^{n+1}$ is in $\calA(b)$ for all $i\in\calV$. This proves
  that the low-order parabolic stage $\bsfu^{\textup{h},n}\mapsto
  \bsfu^{n+1}$ is IDP.

  Summing \eqref{EM:parabolic_update} and \eqref{ER:parabolic_update}, summing over $i\in\calV$, and using
  $\sfk^{\textup{h},n}_{ij}(\sfT^*)=\sfk^{\textup{h},n}_{ji}(\sfT^*)$ together with $\sum_{j\in\calV(i)} \sfk^{\textup{h},n}_{ij}(\sfT^*)=0$ from \eqref{ass:on_kij},
  we obtain
  \[
    \sum_{i\in\calV} m_i\big(\sfe_i^{n+1}+\sfE\lo{r}^{n+1}\big)=\sum_{i\in\calV} m_i\big(\sfe_i^{\textup{h},n}+\sfE\lo{r}^{\textup{h},n}\big).
  \]
  Since the density and the momentum are unchanged in the parabolic stage, \ie $\frac12\varrho_i^{n+1}
    \|\bsfv_i^{n+1}\|_{\ell^2}^2 =\frac12\varrho_i^{\textup{h},n}
    \|\bsfv_i^{\textup{h},n}\|_{\ell^2}^2$, this implies that
  the total energy is conserved.  This completes the proof.
\end{proof}

Now the key question that we have to address is how $\sfT^*_i$ should
be estimated. At low the Mach numbers, it is well known that just
using $\sfT^*_i=\sfT^{n}$ or $\sfT^*_i=\sfT^{\textup{h},n}$ is
sufficient in the sense that this simple choice does not restrict too
much the time step. But this is no longer the case at large Mach
numbers. Hence, similarly to
\citep{knoll_rider_Olson_1999,knoll_chacon_margolin_Mousseau_JCP_2003}
we propose to use an iterative process to estimate $\sfT^*$ that is
robust with respect to the Mach number. But contrary to what is
usually done in the literature, we do not solve the coupled problem
\eqref{parabolic_update} using a Newton-Krylov method. We instead have
observed that using a fixed-point Picard iteration method is
sufficient to solve \eqref{parabolic_update}, even at very high Mach
numbers. The algorithm that we propose proceeds as follows: \textup{(i)} initialize
the process with $\sfT^*_i=\sfT^n_i$; (Do not use $\sfT^*_i=\sfT^{\textup{h},n}_i$.
Robustness is lost by using
$\sfT^*_i=\sfT^{\textup{h},n}_i$ since at steady state $\sfT^{n+1}_i=\sfT^n_i\ne
\sfT^{\textup{h},n}_i$.); \textup{(ii)} Compute the update $\sfE^{n+1}$ by solving
\eqref{ER:proof:lem:IDP_stage_3}; \textup{(iii)} Then update $\sfT^*_i$ for all $i\in\calV$
by solving the nonlinear equation
\begin{equation}
 m_i\varrho_i^{\textup{h},n}\cv^{\textup{h},n}(\sfT_i^{*})(\sfT_i^{*}-\sfT_i^{\textup{h},n})
    = - \dt m_i\sigmaa^{\textup{h},n}(\sfT^*)_i c\, (a\lorr[\sfT_i^{*}]^4 -
    \sfE_{\sfr,i}^{n+1}).\label{Newton_for_T*}
\end{equation}
This can be done with Newton's algorithm using the current
value of $\sfT^*_i$ as initial guess;
\textup{(iv)} Repeat steps
\textup{(ii)}-\textup{(iii)} until some tolerance is achieved.
Finally,  update $T_i^{n+1}$ for all
$i\in\calV$ using \eqref{T:prof:lem:IDP_stage_3} to ensure conservation of the total energy.
A detailed version of the algorithm is shown in
Algorithm~\ref{alg:estimation_of_Tstar}.

\begin{algorithm}[!htb]
  \begin{algorithmic}
    \Require $\bsfu^{\textup{h},n}$, $\sfT^n$
    \State Initialize: $\sfT^*=\sfT^n$; $\sigma\lo{a,ref}$; $\sfE\lo{r,ref}$;   $\epsilon$; 
    \State err\eqq $10^{30}$;
    $\epsilon\up{N} \eqq \epsilon\CROSS \dt\CROSS \sigma\lo{a,ref}\CROSS
    c\CROSS \sfE\lo{r,ref}$
    \While{err > $\epsilon$}
    \State Update $\sfE\lo{r}^{n+1}$ by solving
    \eqref{ER:proof:lem:IDP_stage_3}
    \State $\sfT^{*,\textup{old}}=\sfT^{*}$
    \For{$i\in\calV$}
      \State Let $\sfT^{*}_i$ solve \eqref{Newton_for_T*} up to residual tolerance $\epsilon\up{N}$ (Newton algorithm)
    \EndFor
    \State err = $\|\sfT^*-\sfT^{*,\textup{old}}\|_{\ell^1}/\|\sfT^{*,\textup{old}}\|_{\ell^1}$
    \EndWhile
    \State Update $\sfT^{n+1}$ using \eqref{T:prof:lem:IDP_stage_3} with $\sfT^{*,\textup{old}}$.
    \State Update $\varrho^{n+1}$, $\bsfm^{n+1}$ and $\sfE\lo{m}$ using \eqref{end_of_parabolic_update}
	\end{algorithmic}
\caption{Parabolic update $(\sfT^{n+1},\sfE^{n+1})$}\label{alg:estimation_of_Tstar}
\end{algorithm}

\subsection{Conclusion}

Combining Theorem~\ref{thm:hyperbolic_stage_is_IDP_and_conservative}
with Lemma~\ref{lem:IDP_stage_3} we have proved the following result.

\begin{theorem}\label{Thm:low_order_IDP}
Assume that the space discretization meets the structural assumptions
\eqref{ass1}--\eqref{ass5} from
Assumption~\ref{Ass:assumption_on_space_discretization}.  Assume that
$\bsfu_i^{n}$ is in $\calA(b)$ for all $i\in\calV$. Assume that the time step is
chosen so that $\dt\le \frac{\dt_1\dt_2}{\dt_1+\dt_2}$ with $\dt_1$
and $\dt_2$ defined in \eqref{definition_of_dt1_and_dt2}. Let $\bsfu^{h,n}$
be defined in~\eqref{final_hyperbolic_update}.  Let $\epsilon>0$ and $\bu^{n+1}$ be
defined by Algorithm~\ref{alg:estimation_of_Tstar}
using \eqref{alg:estimation_of_Tstar}-\eqref{ER:proof:lem:IDP_stage_3}-\eqref{Newton_for_T*}.
  Then the
two stage algorithm $\bu^{n}\mapsto
\bu^{\textup{h},n}\mapsto\bu^{n+1}$ has the following properties for all $\epsilon>0$:
\begin{enumerate}[font=\upshape,label=(\roman*)]
\item It is IDP.
\item It is conservative.
\end{enumerate}
\end{theorem}
\begin{proof}
  Owing to the time step restriction $\dt\le
  \frac{\dt_1\dt_2}{\dt_1+\dt_2}$ and
  Theorem~\ref{thm:hyperbolic_stage_is_IDP_and_conservative}, the
  mapping $\bu^{n}\mapsto \bu^{\textup{h},n}$ is IDP (in addition to
  being conservative). Moreover, we have established in
  Lemma~\ref{lem:IDP_stage_3} that
  $\bu^{\textup{h},n}\mapsto\bu^{n+1}$ is IDP. Let now prove that the
  mapping $\bu^{\textup{h},n}\mapsto\bu^{n+1}$ is also
  conservative. As the mass and momentum are unchanged in the
  parabolic step, we just have to prove that $\sum_{i\in\calV}m_i
  \sfE_{\textup{tot},i}^{n+1}=
  \sum_{i\in\calV}m_i\sfE_{\textup{tot},i}^{\textup{h},n}$ if there is
  no energy influx at the boundary. Let $\sfT_i^{*,\textup{old}}$ be the
  penultimate temperature defined in
  Algorithm~\ref{alg:estimation_of_Tstar}. Adding
  $m_i\CROSS$\eqref{EM:parabolic_update} and
  \eqref{ER:parabolic_update}, using that $\varrho^{n+1}\eqq
  \varrho^{\textup{h},n}$, $\bsfv_i^{n+1}\eqq\bsfv_i^{\textup{h},n}$,
  and using the definition of $\sfE_{\textup{m},i}^{n+1}$ in
  \eqref{end_of_parabolic_update}, we obtain
\begin{align*}
 -\dt(\polK(\rho^{\textup{h},n},&T^{*,\textup{old}})\sfE^{n+1})_i= m_i\big(\sfE_{\sfr,i}^{n+1}\!+\varrho_i^{\textup{h},n}\cv^{\textup{h},n}(\sfT_i^{*,\textup{old}}) \sfT_i^{n+1}
 \\
 & \qquad -\sfE_{\sfr,i}^{\textup{h},n}-\varrho_i^{\textup{h},n}\cv^{\textup{h},n}(\sfT_i^{\textup{h},n})\sfT_i^{\textup{h},n}\big)\\
 & =  m_i\Big(\sfE_{\sfr,i}^{n+1}\!+\varrho_i^{n+1}\cv(\varrho_i^{n+1},\sfT_i^{*,\textup{old}}) \sfT_i^{n+1}
 + e_{\textup{cold}}(\varrho_i^{n+1}) + \tfrac12 \|\bsfv_i^{n+1}\|_{\ell}^2
 \\
 & \qquad -\big(\sfE_{\sfr,i}^{\textup{h},n}+\varrho_i^{\textup{h},n}\cv(\varrho_i^{\textup{h},n},\sfT_i^{\textup{h},n})\sfT_i^{\textup{h},n}
 + e_{\textup{cold}}(\varrho_i^{\textup{h},n}) + \tfrac12 \|\bsfv_i^{\textup{h},n}\|_{\ell}^2\big)\Big)\\
 & = m_i(\sfE_{\textup{r},i}^{n+1}+\sfE_{\textup{m},i}^{n+1}-\sfE_{\textup{tot},i}^{\textup{h},n}-\sfE_{\textup{m},i}^{\textup{h},n})
 =m_i(\sfE_{\textup{tot},i}^{n+1}-\sfE_{\textup{tot},i}^{\textup{h},n})
\end{align*}
Hence, if $\sum_{i\in\calV}m_i (\polK(\rho^{\textup{h},n},T^{*,\textup{old}})\sfE^{n+1})_i = 0$, \ie there is no energy influx at the boundary,
the total energy is conserved $\sum_{i\in\calV}m_i \sfE_{\textup{tot},i}^{n+1}= \sum_{i\in\calV}m_i\sfE_{\textup{tot},i}^{\textup{h},n}$ thereby proving that
$\sum_{i\in\calV}m_i \sfE_{\textup{tot},i}^{n+1}= \sum_{i\in\calV}m_i\sfE_{\textup{tot},i}^{n}$.  
\end{proof}


\section{Numerical results}
We now verify that the first-order, conservative IDP
approximation of the model~\eqref{eq:radiation} presented in the paper performs
as advertised.

\subsection{Preliminaries}\label{sec:numerical_prelim}
The numerical tests are performed with two separate codes to verify
reproducibility The first code, henceforth called \JLGcode, is written
in Fortran 95/2003 and does not use any particular software.  It uses
meshes composed of simplices (triangles in 2D and tetraedron in 3D).
The second is a high-performance code, \texttt{ryujin} (henceforth
referred to as \ETcode), see \citep{ryujin-2021-1,ryujin-2021-3},
built upon the \texttt{deal.II} finite element
library~\citep{dealII95}. It is written in C++ and uses cuboids
(quadrangles in 2D and hexahedrons in 3D). Both codes reproduce the
algorithm described in the paper using the additive splitting
described in
\eqref{definition_of_dt1_and_dt2}--\eqref{final_hyperbolic_update} for
the explicit hyperbolic stage followed by the implicit parabolic stage
described in Algorithm~\ref{alg:estimation_of_Tstar}.  The space
approximation in \JLGcode is done with continuous $\polP_1$
finite elements.  The space approximation in \ETcode is done
with continuous $\polQ_1$ finite elements.
The time step $\dt$ is systematically computed in both codes with
\eqref{def_of_dt_and_theta} using the definition
\begin{equation}
  \dt  \eqq \cfl
\frac{\dt_1\dt_2}{\dt_1+\dt_2}.
  \end{equation}
where $\cfl\in (0,1]$ is the user-dependent Courant–Friedrichs–Lewy
  number.  Unless specified otherwise the relative tolerance in
  Algorithm~\ref{alg:estimation_of_Tstar} is set to $\epsilon=10^{-5}$.
  
In all the tests for which an analytical solutions exists, we compute the error at time $t$
as follows:
\begin{multline}\label{def_l1_cumulative_eror}
  \text{err}(t) \eqq
  \tfrac{\|\rho_h(\cdot,t)-\rho(\cdot,t)\|_{L^1(\Dom)}}{\|\rho(\cdot,T)\|_{L^1(\Dom)}}
  + \tfrac{\|\bbm_h(\cdot,t)-\bbm(\cdot,t)\|_{\bL^1(\Dom)}}{\|\bbm(\cdot,T)\|_{\bL^1(\Dom)}}\\
  + \tfrac{\|E_{\textup{m}h}(\cdot,t)-E\lo{m}(\cdot,t)\|_{L^1(\Dom)}}{\|E_{\textup{m}}(\cdot,t)\|_{L^1(\Dom)}}
  + \tfrac{\|E_{\textup{r}h}(\cdot,t)-E\lo{m}(\cdot,t)\|_{L^1(\Dom)}}{\|E_{\textup{r}}(\cdot,t)\|_{L^1(\Dom)}},
  \end{multline}
where $\rho_h$, $\bbm_h$, $E_{\textup{m}h}$, and $E_{\textup{r}h}$ are
the approximate density, momentum, mechanical energy, and radiation
energy, and $\rho$, $\bbm$, $E_{\textup{m}}$, and $E_{\textup{r}}$ are
the exact density, momentum, mechanical energy, and radiation energy.

\subsection{Units}\label{sec:units}
Although the system \eqref{eq:radiation} can be made non-dimensional
by proceeding as in \cite{BHEML:17}, we are going to follow the
literature and report results using dimensional quantities.  Unless
stated otherwise, we use the following units. Length scales and
distances are measured in $\unit{cm}$.  Scattering and absorption
cross sections are measured in $\unit{cm^{-1}}$. Masses are measured
in $\unit{g}$. Time is measured in shake $\unit{sh}$ (called shakes
for plural). One shake \unit{sh} is equal to $\SI{1e-8}{s}$.  Energies
are measured in \unit{GJ}.  Recall that one giga-Joule is equal to
$\SI{1e9}{J}$. Pressures are measured in $\unit{GJ\per cm^3}$; recall that
one $\unit{GJ\per cm^3}$ is equal to $10\,\unit{Gbar}$.

The temperatures are rescaled by the Boltzmann
constant $k$, \ie we use $\widetilde T\eqq kT$ instead of $T$, and the
temperatures thus rescaled are measured in kiloelectronvolt
$\unit{keV}$. Recall that one $\unit{eV}$ is also an energy unit and
$\SI{1}{eV}\eqq\SI{1.602176634e-19}{J}$; hence, the ratio
$\SI{1}{J}/\SI{1}{eV}$ is dimensionless.

The radiation constant
$a\lo{r}\eqq \frac{4\sigma}{c}$ is also rescaled, instead of using
$a\lo{r}$ we use $\widetilde a\lo{r} \eqq a\lo{r}/k^4$. The rescaled radiation constant
is measured in $\unit{GJ\per cm^3 keV^4}$.

The specific heat
capacities at constant volume is also rescaled by the Boltzmann
constant, \ie we use $\widetilde\cv\eqq \cv/k$ instead of $\cv$, and
the rescaled specific heat capacity is measured in $\unit{GJ\per (g.keV)}$
where $\unit{GJ}$ is the giga-Joule unit.
We recall that for ideal gases the heat capacity is given by
$\cv=\frac{1}{\gamma-1}k \text{Na}\frac{Z_{\text{eff}}+1}{A}$, where
$\gamma$ is the heat capacity ratio, $\text{Na}$ is the Avogadro
number, $Z_{\text{eff}}$ is effective nuclear charge (also called
effective ionization state) and $A$ is the atomic mass. We use the ideal
gas equation of state in all the tests
reported in the paper, and we arbitrarily
choose $\frac{Z_{\text{eff}}+1}{A}$ so that $\widetilde\cv
=\SI{0.15}{GJ\per keV.g}$, and unless stated otherwise,
we use $\gamma= \frac53$.

The constants used in the following numerical tests are reported in
Table~\ref{Tab:units_and_constants}.
\begin{table}[h]\centering
  \caption{Units}\label{Tab:units_and_constants}
  \begin{tabular}{ll}
    speed of light $c$                           & $\SI{2.99792458e2}{cm\per\unit{sh}}$     \\
    rescaled radiation constant $\widetilde a\lo{r}$ & $\SI{1.3720172e-2}{GJ\per (cm^3.keV^4)}$ \\
    heat capacity ratio $\gamma$                 & $\frac{5}{3}$ or $\gamma=1.2$ for Mach 50 test
    \\
    rescaled   heat capacity $\widetilde\cv$     & \SI{0.15}{GJ\per keV.g}                  \\\hline
  \end{tabular}
\end{table}
The speed of light and the rescaled radiation constant reported
therein are copied verbatim from the \texttt{ExactPack}
software~\citep{ExactPack_url}, \cite{ExactPack_article}.

\subsection{Marshak wave} We start by considering a simplified version of the problem \eqref{eq:radiation}
to verify the correctness of the approximation of the parabolic
stage. As in \cite[Eq. (1)\&(5)]{Pomraning_1979}, we neglect the fluid
motion and solve the system
\begin{subequations}\label{eq:only_radiation}\begin{align} 
     & \partial_t(\rho \cv T)
    = -\sigmaa c(a\lorr \Tmech(\bu)^4-\ER), \label{EM:eq:only_radiation}             \\
     & \partial_t\ER+
    - \DIV(\tfrac{c}{3\sigmat}\GRAD\ER) = \sigmaa c(a\lorr \Tmech(\bu)^4-\ER), \label{ER:eq:only_radiation}
\end{align}\end{subequations}
where the density $\rho$ is constant. We use the same setting as in
\cite[\S7]{Larsen_JCP_2013}.  The computational domain is
$\Dom\eqq(0,\ell_\Dom)$ with $\ell_\Dom=\SI{0.025}{cm}$. We take
$\rho=\SI{2}{g \per cm^3}$, and the constant $\widetilde\cv$,
$\widetilde a\lo{r}$ and $c$ are given in
Table~\ref{Tab:units_and_constants}. We use $\sigma_t=\sigma_a=300
(\frac{T\lo{ref}}{T})^3$ with $T\lo{ref}=\SI{1}{keV}$. The initial
data are $T_0=\SI{0.01}{keV}$ and $E_{\textup{r},0}=\widetilde a\lo{r}
T_0^4$.  We enforce the Dirichlet boundary $\ER(0,t) = \widetilde
a\lo{r} T\lo{ref}^4$ and the homogeneous Neumann boundary condition
$\partial_x \ER(\ell_\Dom,t)=0$ for all $t>0$.  The approximation is
done by using the algorithm described in the paper without invoking
the hyperbolic update (\ie at the beginning of
\S\ref{sec:approx_stage_3} we set
$\bu^{\textup{h},n}\eqq(\varrho,\bzero,\sfE_{\sfm}^{n},\sfE_{\sfr}^{n})\tr$).

The composite relative error in the $L^1$-norm (defined in
\eqref{def_l1_cumulative_eror}) is computed with an approximation of
the exact solution at the final time $t=\SI{0.02}{sh}$ using \JLGcode
on a uniform mesh composed of $20\,001$ grid points. We show in
Table~\eqref{Tab:Marshak_wave_error_table} the composite relative
$L^1$-error for 6 uniform meshes.  We use $\cfl=0.25$ in all the
simulation (the algorithm is $L^2$-stable irrespective of the value of
$\cfl$). We observe first
order convergence in the asymptotic regime for both \JLGcode and
\ETcode.
\begin{table}[htbp!]
  \centering
  \begin{tabular}[b]{rrlrl}
    \multicolumn{3}{c}{\JLGcode} \\
    \toprule
    $I$   &   $L^1$-\text{error} & rate \\
     65 & \num{2.95E-03} &  -- \\ 
    129 & \num{2.27E-03} &0.38\\ 
    257 & \num{1.52E-03} &0.58\\ 
    513 & \num{9.11E-04} &0.74\\ 
   1025 & \num{5.00E-04} &0.87\\ 
   2049 & \num{2.83E-04} &0.82\\ 
  \end{tabular}
  \begin{tabular}[b]{rrlrl}
    \multicolumn{3}{c}{\ETcode} \\
    \toprule
    $I$   &   $L^1$-\text{error} & rate \\
     65 & \num{3.94E-03} &  -- \\ 
    129 & \num{4.27E-03} &-.11\\ 
    257 & \num{2.62E-03} &0.71\\ 
    513 & \num{1.57E-03} &0.73\\ 
   1025 & \num{6.36E-04} &1.31\\ 
   2049 & \num{3.47E-04} &0.88\\
  \end{tabular} 
  \caption{Marshak wave problem ~\eqref{eq:only_radiation}. Exact solution approximated with \JLGcode and $20\,001$ grid points.
  Composite relative $L^1$-error at $t=\SI{0.02}{sh}$ for 6 meshes with $\cfl=0.25$.}  \label{Tab:Marshak_wave_error_table}
\end{table}
We plot the solution profiles using \ETcode for the Marshak wave for the temperature (left) and radiation energy (right).
\begin{figure}[htbp!]
  \centering\vspace{-\baselineskip}
  \includegraphics[width=0.325\textwidth,trim=5 5 0 5,clip]{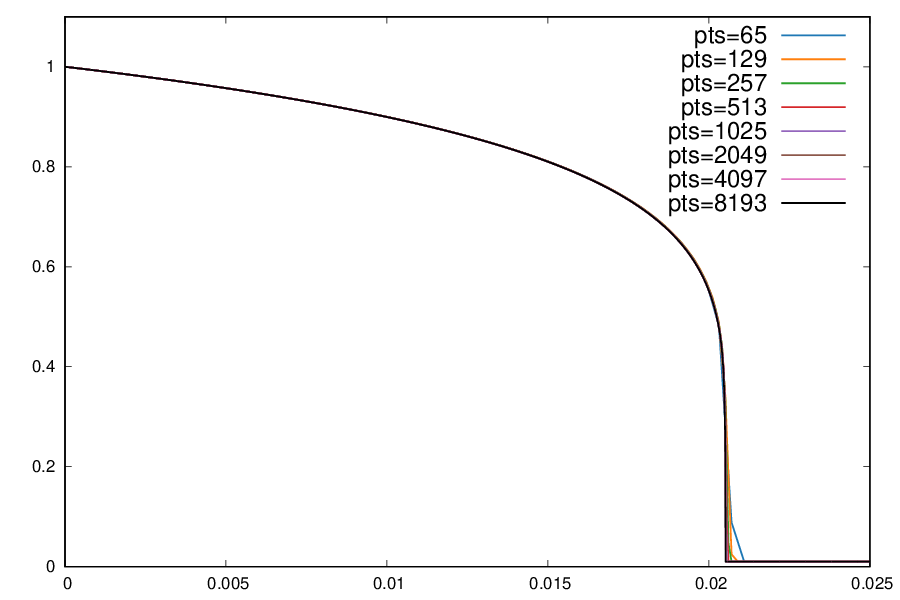}
  \includegraphics[width=0.325\textwidth,trim=5 5 0 5,clip]{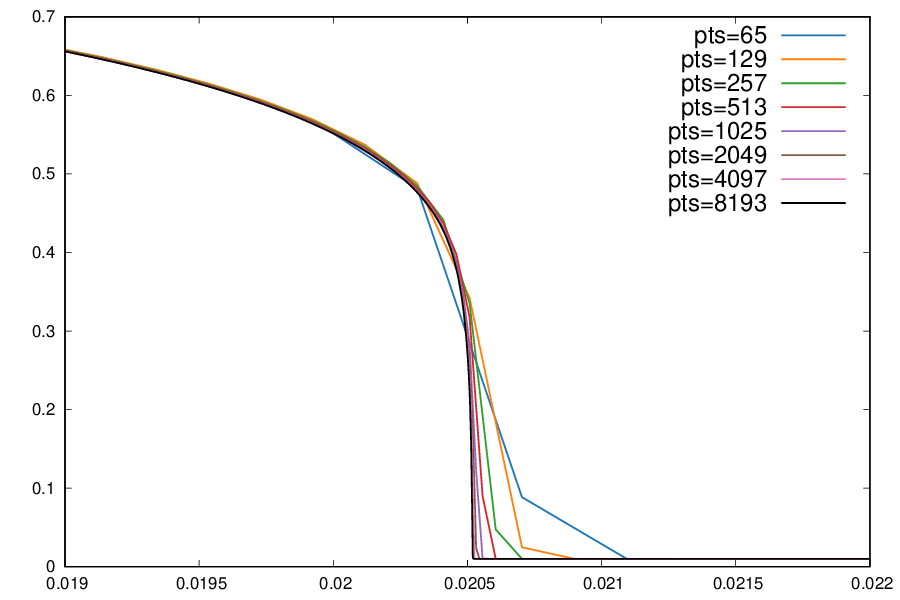}
  \includegraphics[width=0.325\textwidth,trim=5 5 0 5,clip]{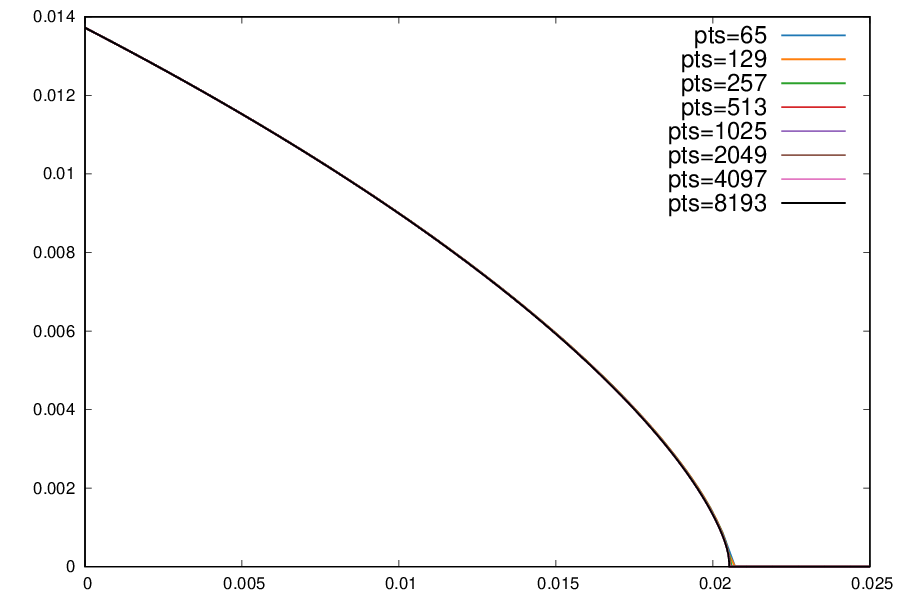}
  \caption{\label{fig:ryujin_marshak} 
    Marshak wave with  $\sigma_t=300(\frac{T\lo{ref}}{T})^3$ at $T=\mynum{0.02}{sh}$. Left: temperature.
    Center: zoom in on temperature shock location. Right: radiation energy.}
\end{figure}

\subsection{1D radiative shocks}
We now consider common steady radiative shock configurations found in
the literature (see~\cite{lowrie2008radiative,delchini2015entropy}).
All the tests reported in this section are performed with
\JLGcode.  We have verified that \ETcode gives the
same results (not reported for brevity).  For all
the configurations the reference density is
$\rho_{\text{ref}} = \mynum{1}{g~cm^{-3}}$ and the reference
temperature is $\widetilde T_{\text{ref}} = \mynum{0.1}{keV}$. Given
$\gamma$, $\rho_{\text{ref}}$, $\widetilde T_{\text{ref}}$ and a
prescribed Mach number, the semi-analytic steady radiative shock
solution is computed with the methodology described
in~\citep{lowrie2008radiative} using the \texttt{ExactPack}
software~\citep{ExactPack_url,ExactPack_article} developed at Los
Alamos National Laboratory.  To shorten the time to reach steady state we initialize
every simulation by interpolating the semi-analytic solution obtained
from \texttt{ExactPack} using $10^6$ uniform grid points.  Unless
stated otherwise, we run all the tests with the
Courant–Friedrichs–Lewy number $\cfl=1$.

\subsubsection{Subcritical tests}

We first consider the Mach 1.2 and Mach 3 radiative shock cases which
are categorized as ``subcritical'' in~\citep{lowrie2008radiative}.
The computational domain is $\Dom = (-\mynum{0.02}{},
\mynum{0.02}{cm})$.  Dirichlet boundary conditions are enforced. The
tests are performed on a sequence of uniform meshes.  We assume that
$\sigma_t = \mynum{500}{cm^{-1}}$ and $\sigma_a =
\mynum{500}{cm^{-1}}$, \ie $\sigma_s = 0$.  The final time is set to
$t = \mynum{1}{sh}$. This time is long enough for steady state to be
reached.

\begin{table}[htbp!]
  \centering
  \begin{tabular}[b]{rrlrl}
    \multicolumn{5}{c}{$\sigma_a = \sigma_t = \mynum{500}{cm^{-1}}$}\\
    \toprule
    $I$   & Mach 1.2                     &      & Mach 3   &      \\
    101   & \numm{0.08141660177673432}   &      & \numm{0.0830743876797732}    &      \\
    201   & \numm{0.04157833817241097}   & 0.97 & \numm{0.04206934854045891}   & 0.98 \\
    401   & \numm{0.02054414451572962}   & 1.02 & \numm{0.02097120323124426}   & 1.00  \\
    801   & \numm{0.01011624843851955}   & 1.02 & \numm{0.01031794569695309}   & 1.02 \\
    1601  & \numm{0.005010593968288955}  & 1.01 & \numm{0.005034324215589092}  & 1.04 \\
    3201  & \numm{0.002499508129445163}  & 1.00  & \numm{0.002453191299661565} & 1.04 \\
    6401  & \numm{0.001253661463923093}  & 1.00  & \numm{0.00120910406456733}  & 1.02 \\
    12801 & \numm{0.0006309372563796663} & 0.99 & \numm{0.0006180026406430259} & 0.97 \\
    \bottomrule
  \end{tabular}\hfil
  \begin{tabular}[b]{rrl}
    \multicolumn{3}{c}{$\sigma_a = \sigma_t = 500 \frac{\rho T_\text{ref}^{3.5}}{\rho_\text{ref} T^{3.5}}\mynum{}{cm^{-1}}$}\\
     \toprule
     $I$   & Mach 3   & \\
    101 & \num{7.83E-02} &  -- \\ 
    201 & \num{3.91E-02} &1.01\\ 
    401 & \num{1.89E-02} &1.05\\ 
    801 & \num{9.14E-03} &1.05\\ 
   1601 & \num{4.54E-03} &1.01\\ 
   3201 & \num{2.25E-03} &1.01\\ 
   6401 & \num{1.10E-03} &1.03\\ 
  12801 & \num{5.30E-04} &1.05\\ 
   \bottomrule
      \end{tabular}
   \caption{Left table: $L^1$ errors and convergence rates for 1D radiative
     shock with $\sigma_a = \mynum{500}{cm^{-1}}$ and $\sigma_s = 0$
     for Mach 1.2 and 3 respectively.  Right table: $L^1$ errors and
     convergence rates for 1D radiative shock with $\sigma_a=500
     \frac{\rho}{\rho_\text{ref}}\left(\frac{T_\text{ref}}{T}\right)^{3.5}\mynum{}{cm^{-1}}$
     and $\sigma_s = 0$ for Mach 3.}\label{tab:rad_shock_case_1}
\end{table}

\begin{figure}[htbp!]
  \centering\vspace{-\baselineskip}
  \includegraphics[width=0.31\textwidth,trim=30 10 15 10,clip=]{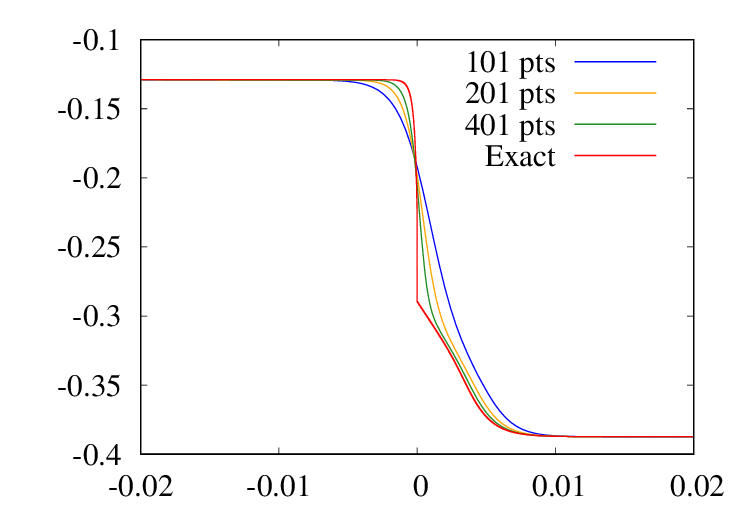}
  \includegraphics[width=0.31\textwidth,trim=30 10 15 10,clip=]{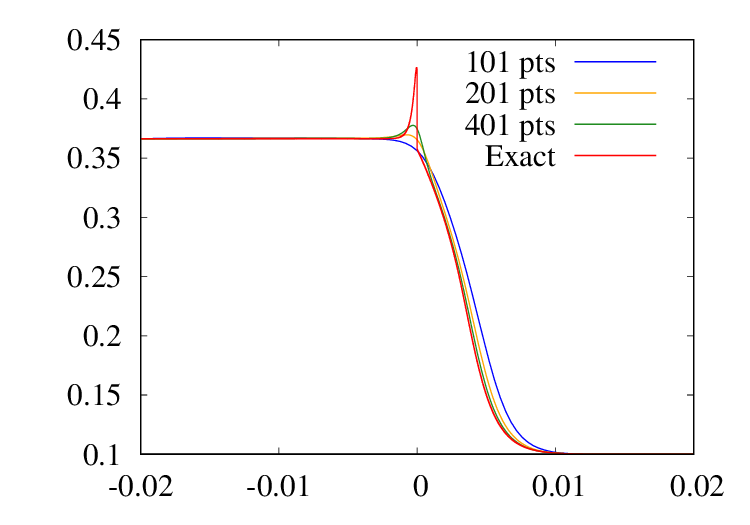}
  \includegraphics[width=0.31\textwidth,trim=30 10 15 10,clip=]{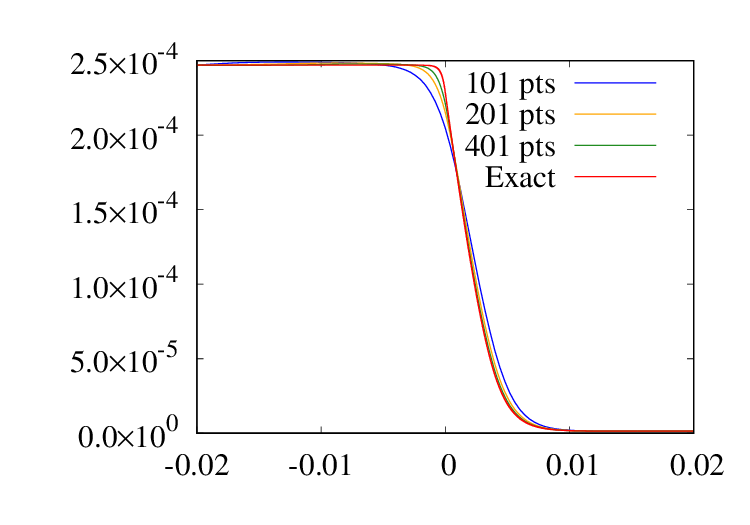}
  \caption{Mach 3 radiative shock with $\sigma_t =
    \mynum{500}{{cm}^{-1}}$ at $T=\mynum{1}{sh}$. Left: velocity.
    Center: temperature. Right: radiation energy.}
   \label{fig:rad_shock_case_1}
\end{figure}

In Table~\ref{tab:rad_shock_case_1}, we report the
 cumulative $L^1$-norm error defined in \eqref{def_l1_cumulative_eror}
for Mach numbers 1.2 and 3, respectively.  We observe first-order
rate as expected. In Figure~\ref{fig:rad_shock_case_1}, we plot the
numerical velocity, material temperature and radiation energy
for the Mach 3 configuration using $101$, $201$, and $401$ grid points and we compare
the results to the semi-analytic solution.

 We also report in Table~\ref{tab:rad_shock_case_1} a test done at
 Mach 3 with the opacity depending on $\rho$ and $T$; see \eg
 \cite[\S4.4]{Delchini_Ragusa_Ferguson_IJNMF_2017} or
 \cite[Fig.~17]{lowrie2008radiative}. More specifically; we take
 $\sigma_t=\sigma_a = 500
 \frac{\rho}{\rho_\text{ref}}\left(\frac{T}{T_\text{ref}}\right)^{-3.5}\mynum{}{cm^{-1}}$.
 In this case the computational domain is $\Dom = (-\mynum{0.3}{},
 \mynum{0.3}{cm})$. The simulations are run up to $t=\mynum{10}{sh}$ to reach
 steady state.  The nonlinear dependency of the cross section with
 respect to the density and the temperature makes the zone out of
 thermodynamics equilibrium larger (see temperature peak in the center
 panel of Figure~\ref{fig:mach3_temp} and compare to
 Figure~\ref{fig:rad_shock_case_1}).  The convergence rates are
 reported in the right table in Table~\ref{tab:rad_shock_case_1}. We
 observe first order convergence in this case as well.

\begin{figure}[htbp!]
  \centering
  \includegraphics[width=0.31\textwidth,trim=30 10 15 10,clip=]{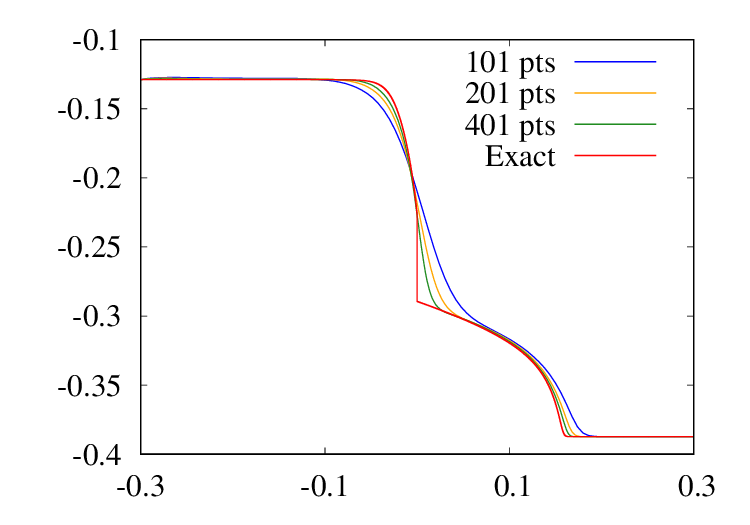}
  \includegraphics[width=0.31\textwidth,trim=30 10 15 10,clip=]{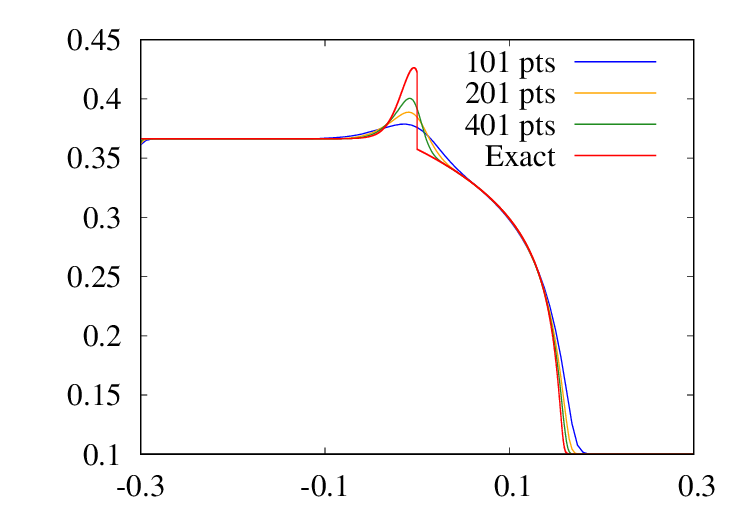}
  \includegraphics[width=0.31\textwidth,trim=30 10 15 10,clip=]{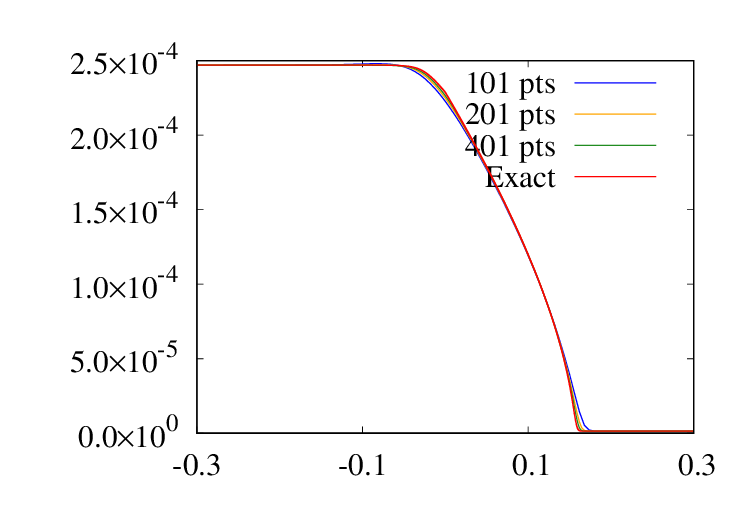}
 \caption{Mach 3 radiative shock with $\sigma_a=500
     \frac{\rho}{\rho_\text{ref}}\left(\frac{T_\text{ref}}{T}\right)^{3.5}\mynum{}{cm^{-1}}$
     and $\sigma_s = 0$ for Mach 3. Left: velocity.
  Center: temperature. Right: radiation energy.} \label{fig:mach3_temp}
\end{figure}

\subsubsection{Supercritical tests}

Now, we consider the Mach 10, Mach 30, and Mach 50 radiative shock
cases which are categorized as ``supercritical''
in~\citep{lowrie2008radiative}.

\begin{table}[htbp!]\centering\addtolength{\tabcolsep}{-2.pt}
\begin{tabular}[b]{rrlrl}
 \multicolumn{5}{c}{$\sigma_a = \sigma_t = \mynum{500}{cm^{-1}}$}\\
    \toprule
    $I$ & Mach 30         &      &  Mach 50         &     \\
    101 & \num{1.58E-01} &  --   &  \num{1.75E-01} &  -- \\
    201 & \num{9.03E-02} &0.81 &  \num{9.51E-02} &0.88\\ 
    401 & \num{4.79E-02} &0.92 &  \num{5.47E-02} &0.80\\ 
    801 & \num{2.61E-02} &0.88 &  \num{3.03E-02} &0.85\\
   1601 & \num{1.39E-02} &0.91 &  \num{1.67E-02} &0.87\\
   3201 & \num{7.30E-03} &0.93 &  \num{9.35E-03} &0.83\\
   6401 & \num{3.79E-03} &0.94 &  \num{5.03E-03} &0.89\\
  12801 & \num{1.98E-03} &0.94 &  \num{2.66E-03} &0.92\\ 
    \bottomrule
\end{tabular}\hfil
\begin{tabular}[b]{rrl}
  \multicolumn{3}{c}{$\sigma_a = \sigma_t = 500 \frac{T_\text{ref}\rho}{\rho_\text{ref} T}\mynum{}{cm^{-1}}$}\\
    \toprule
    $I$   & Mach 10       &       \\
      101 &  \num{1.15E-01}  &  -- \\ 
    201 &  \num{5.76E-02}  & \num1.01\\ 
    401 &  \num{3.05E-02}  & \num0.92\\ 
    801 &  \num{1.62E-02}  & \num0.92\\ 
   1601 &  \num{7.72E-03}  & \num1.07\\ 
   3201 &  \num{3.96E-03}  & \num0.96\\ 
  12801 &  \num{9.96E-04}  & \num1.00\\
    \bottomrule
  \end{tabular}
   \caption{Left table: $L^1$ errors and convergence rates for 1D
     radiative shock with $\sigma_a = \sigma_t = \mynum{500}{cm^{-1}}$
     and $\sigma_s = 0$ for Mach 30 and 50, respectively.  Right
     table: $L^1$ errors and convergence rates for 1D radiative shock
     with $\sigma_a =\sigma_t= 500
     \frac{T_\text{ref}\rho}{\rho_\text{ref} T}\mynum{}{cm^{-1}}$ and
     $\sigma_s = 0$ for Mach 10.}\label{tab:rad_shock_case_2}
\end{table}
The computational domain for the Mach
10 case is $D = (-\mynum{2}{}, \mynum{5}{cm})$ and  we
use density and temperature dependent opacities $\sigma_t=\sigma_a = 500
\frac{\rho}{\rho_\text{ref}}\left(\frac{T}{T_\text{ref}}\right)^{-3.5}\mynum{}{cm^{-1}}$.
The simulation time to reach steady state is $t=\mynum{50}{sh}$.
The computational domain for Mach 30 is $D = (-\mynum{0.1}{},
\mynum{0.4}{cm})$ and we use $\sigma_t=\sigma_a = \mynum{500}{cm^{-1}}$.
The simulation time to reach steady state in this case is $t=\mynum{1}{sh}$.  The domain
for Mach 50 is $D = (-\mynum{0.1}{}, \mynum{0.6}{cm})$ and we use
$\sigma_t=\sigma_a =\mynum{500}{cm^{-1}}$.
The simulation time to reach steady state in this case is $t=\mynum{10}{sh}$.
We use $\gamma=1.2$ for the Mach 50 case only (the solution process
used in the \texttt{ExactPack}
software is ill-posed for $\gamma=\frac53$ at Mach 50). These tests are again performed on a
sequence of uniform meshes. The final time is set to $T =
\mynum{1}{sh}$.  In the left panel of
Table~\ref{tab:rad_shock_case_2}, we report the cumulative $L^1$-norm
error for the tests at Mach numbers 30 and 50. We show in the right
panel of the table the cumulative $L^1$-norm error for the tests at
Mach numbers 10.  In all the case we observe the first-order rate as
expected.

 \begin{figure}[htbp!]
  \centering
  \includegraphics[width=0.31\textwidth,trim=30 10 15 10,clip=]{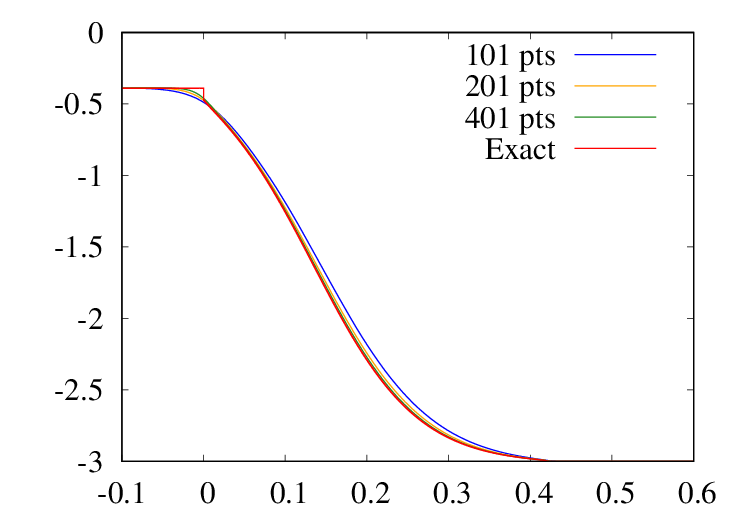}
  \includegraphics[width=0.31\textwidth,trim=30 10 15 10,clip=]{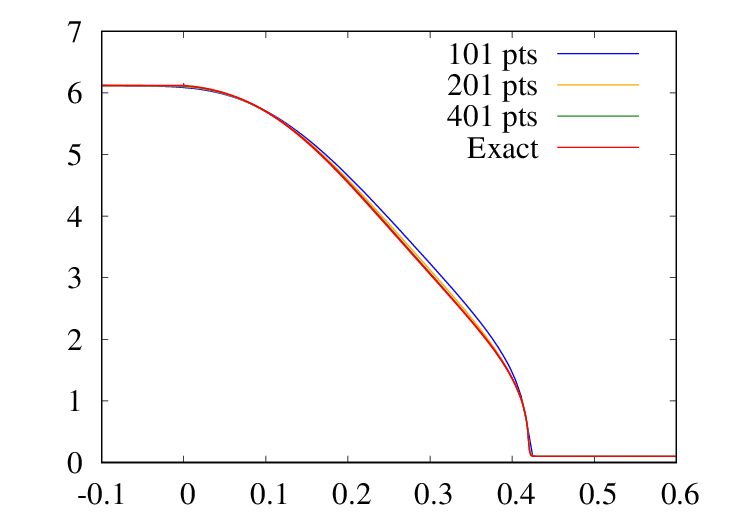}
  \includegraphics[width=0.31\textwidth,trim=30 10 15 10,clip=]{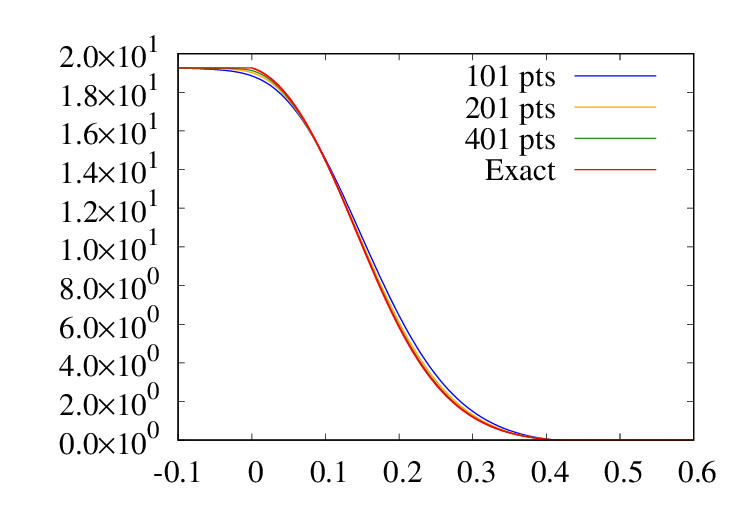}
  \caption{Mach 50 radiative shock with $\sigma_a=\sigma_t =
    \mynum{500}{{cm}^{-1}}$ at $t=\mynum{10}{sh}$. Left: velocity.
    Center: temperature. Right: radiation energy.\label{fig:mach50}}%
\end{figure}
In Figure~\ref{fig:mach50}, we plot the numerical velocity, material temperature and
radiation energy for the Mach 50 configuration.

\subsection{ICF-like configuration}

We finish by simulating a setting loosely inspired from an indirect drive inertial
confinement fusion experiment in one and two dimensions in a hohlraum device,
see \eg \cite{PhysRevLett.117.245001}. Our objective is not to be
close to one particular experiment but instead to give some feeling on
how the method behaves when solving a problem with data that are in a
realistic range.

\subsection{Setting}\label{Sec:ICF_setting}
We model the transverse cross section of the hohlraum by a disk
$\Dom=\{\bx\in\Real^d\st \|\bx\|_{\ell^2}<r\lo{ext}\}$ centered at
$\bzero$ of radius $r\lo{ext}\eqq \mynum{0.3}{cm}$. A spherical pellet
with a high-density carbon shell contains a light material (either gas
or solid deuterium).  The outside of the pellet is composed of a low
pressure gas. In the indirect drive setting considered here, radiation
energy is injected into the domain through the boundary of $\Dom$
using Dirichlet boundary conditions. The slip boundary condition on
the velocity is enforced at the boundary of $\Dom$. We simplify the
setting by modeling all the materials with the same ideal gas equation
of state with $\gamma=\frac53$.  We now summarize the geometry,
initial conditions, and boundary conditions:
\begin{equation}
  \begin{alignedat}{3}
    & r\lo{int}\eqq \SI{0.13}{cm},&&
     r\lo{sh}\eqq \SI{0.15}{cm},&&
     r\lo{ext}\eqq \SI{0.6}{cm}, \\
    &\rho\lo{int}\eqq \SI{0.0005}{\gram\per\cm^3},&&
    \rho\lo{sh}\eqq \SI{3.5}{\gram\per\cm^3},&&
    \rho\lo{ext}\eqq \SI{0.0001}{\gram\per\cm^3},\\
    & T\lo{int}\eqq \SI{2.6d-6}{\kilo\electronvolt},\quad&&
    T\lo{ref}\eqq \SI{0.25}{\kilo\electronvolt},\quad &&
    \gamma = \tfrac{5}{3},\\
    &\rho\lo{ref}\eqq \SI{1}{\gram\per\cm^3},&&
    \sigma\lo{ref} \eqq \SI{5d3}{\per\cm},&&
    \sigma\lo{a}(\rho)=\sigma\lo{t}(\rho) = \tfrac{\rho}{\rho\lo{ref}}
    \sigma\lo{ref}.
    \end{alignedat}\vspace{-\baselineskip}
\end{equation}

\begin{equation}
  \rho_0(\bx), \bu_0(\bx), T_0(\bx), E\lo{r,0}(\bx)=
    \begin{cases}
      \rho\lo{int}, \bzero,  T\lo{int},\quad a\lorr T\lo{int}^4 &  \|\bx\|_{\ell^2}< r\lo{int} \\
      \rho\lo{sh}, \bzero,  \frac{\rho\lo{int}}{\rho\lo{sh}} T\lo{int}, a\lorr (\frac{\rho\lo{int}}{\rho\lo{sh}} T\lo{int})^4  &  \|\bx\|_{\ell^2}< r\lo{sh} \\
      \rho\lo{ext},\bzero, \frac{\rho\lo{int}}{\rho\lo{ext}} T\lo{int}, a\lorr T\lo{ref}^4&  \text{otherwise}.
    \end{cases}\vspace{-\baselineskip}
\end{equation}

\begin{equation}
  \bu\SCAL\bn_{|\front} =0,\quad E_{\textup{r}|\front} = a\lorr T\lo{ref}^4.
  \end{equation}

\subsection{One-dimensional case}
We start by solving the problem in one space dimension.  Although to
account for the spherical nature of the pellet, the problem should be
solved in spherical coordinates, we work with the
Cartesian coordinate system as our objective is just to demonstrate the
robustness of the method.

\begin{figure}[ht]
  \includegraphics[width=0.32\linewidth,trim=30 8 25 10,clip=]{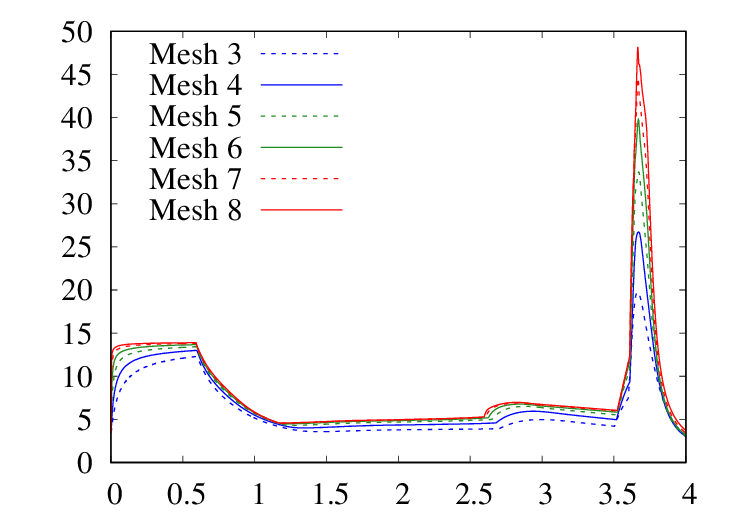}
  \includegraphics[width=0.32\linewidth,trim=30 8 25 10,clip=]{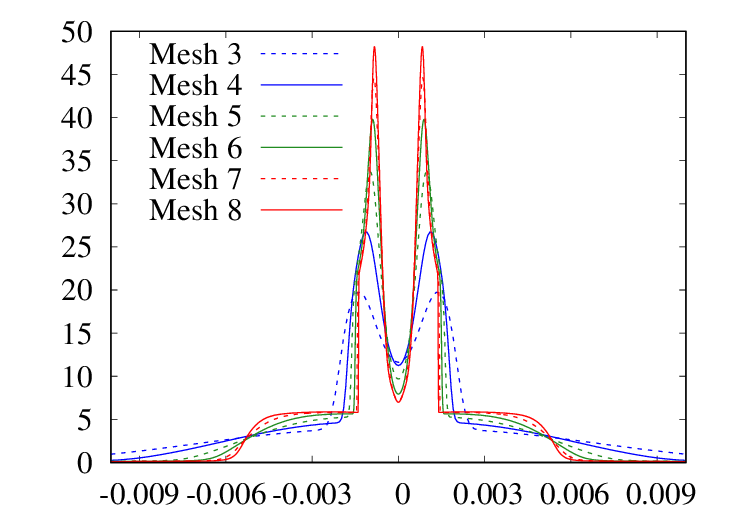}
  \includegraphics[width=0.32\linewidth,trim=30 8 25 10,clip=]{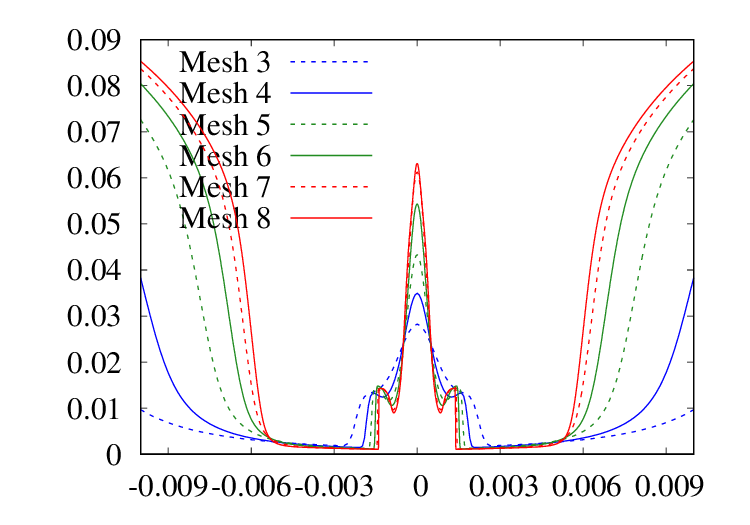}
  \caption{1D ICF. Left: time history of maximum density. Center:
    density profile at the time maximum density is reached.  Right:
    temperature at the time maximum density is reached. Mesh3 has
    $4097$ grid points and Mesh8 has $131073$ grid points.}\label{Fig:1d_ICF}
  \end{figure}

The computational domain is $\Dom=(-r\lo{ext},r\lo{ext})$ with
$r\lo{ext}\eqq \mynum{0.6}{cm}$.  The problem is solved on various
uniform meshes with increasing number of grid points to verify that
convergence occurs. The simulation time is $t=\mynum{4}{sh}$. We show
in Figure~\ref{Fig:1d_ICF} the time history of the maximum density for
six meshes $(\text{mesh}i)_{i\in\{3:8\}}$ with number of grid points
equal to $2^{9+i}$ using \JLGcode. We observe in the left panel of the
figure that the maximum density rises very quickly after
initialization due to a very strong compression wave crossing the
high-density carbon (this wave is clearly visible in the left panel in
Figure~\ref{fig0p25_density}). The time for this wave to cross the
high-density carbon is approximately $\mynum{0.6}{sh}$. The left and
right compression waves then travel in the interior material and make
contact at approximately $\mynum{2.15}{sh}$. At this time the
compression process starts and reaches it maximum at about
$\mynum{3.67}{sh}$. By inspecting the left panel we observe that this
time depends very little of the mesh resolution, but the actual value
of the maximum density at this time can only be well captured on very
fine grids. We show in the center and right panels of
Figure~\ref{Fig:1d_ICF} closeup views of the density and temperature
fields in the interval $(-0.01,0.01)\mynum{}{cm}$ at the time
$3.67\mynum{}{sh}$ when the density peak reaches its maximum. These
results are well reproduced with \ETcode and are therefore not
reported for brevity.

\subsection{Two-dimensional case}
We finish with two-dimensional simulations of the ICF problem keeping
the setting described in \S\ref{Sec:ICF_setting}. Again, the problem should be solved
in cylindrical coordinates, but we use Cartesian coordinates for
simplicity.

\begin{figure}[ht]
  \includegraphics[width=0.32\linewidth,trim=30 8 25 10,clip=]{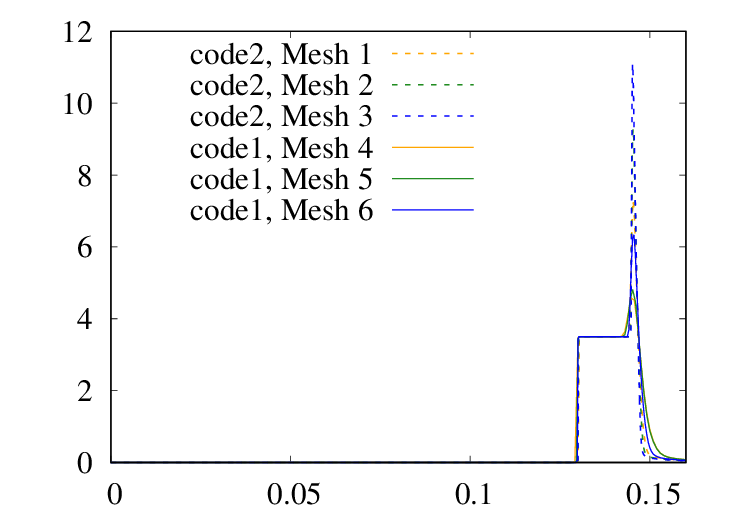}
  \includegraphics[width=0.32\linewidth,trim=30 8 25 10,clip=]{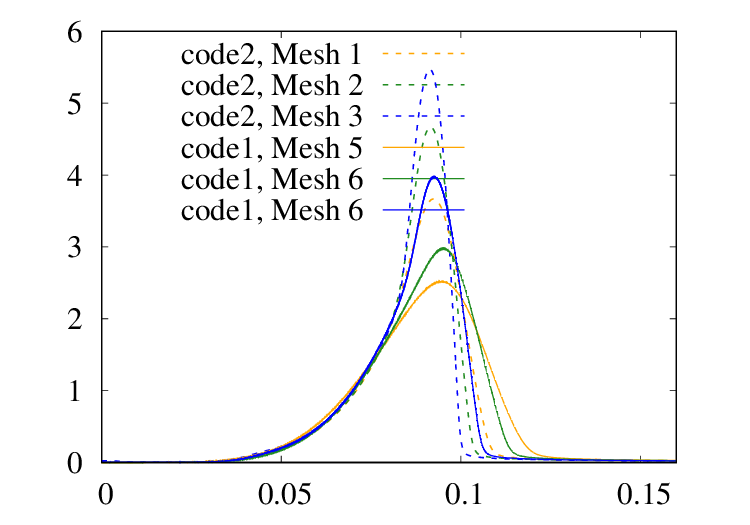}
  \includegraphics[width=0.32\linewidth,trim=30 8 25 10,clip=]{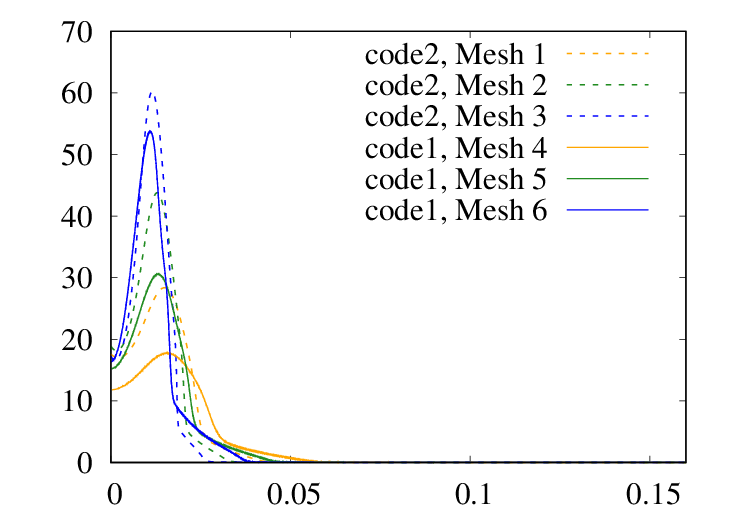}
    \caption{Density \vs $r$. Left to right: $t=\mynum{0.1}{sh}$;
    $t=\mynum{2.0}{sh}$; $t=\mynum{3.7}{sh}$.}\label{fig0p25_density}
\end{figure}

\begin{figure}[ht]
  \includegraphics[width=0.32\linewidth,trim=30 8 25 10,clip=]{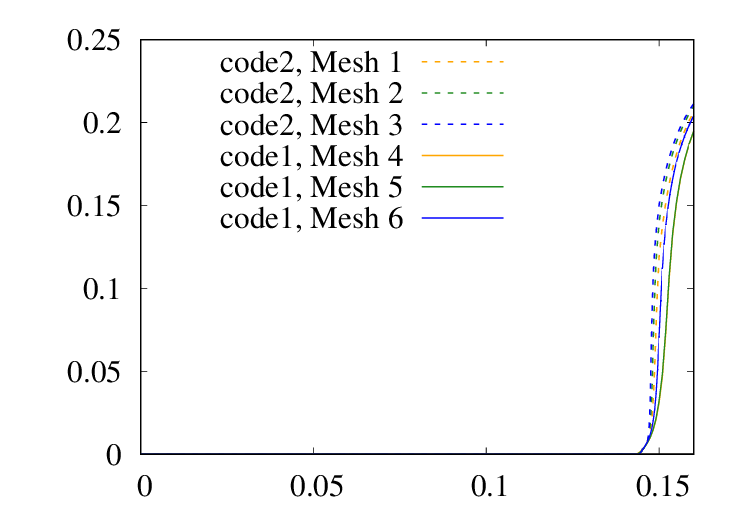}
  \includegraphics[width=0.32\linewidth,trim=30 8 25 10,clip=]{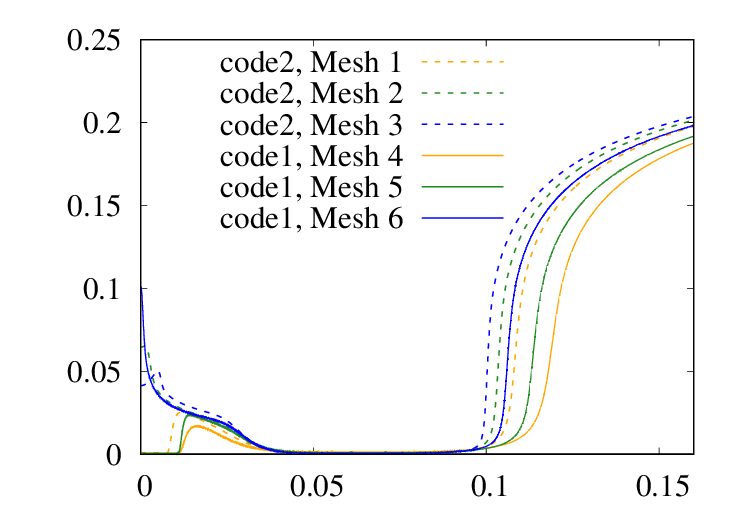}
  \includegraphics[width=0.32\linewidth,trim=30 8 25 10,clip=]{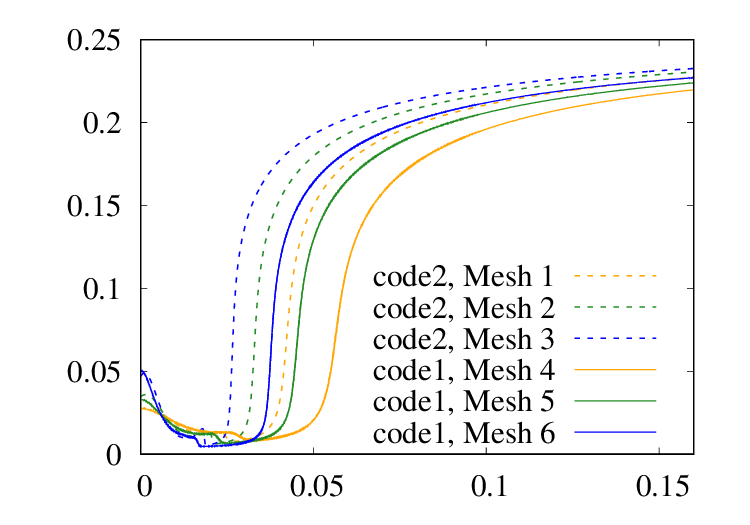}
  \caption{Temperature \vs $r$. Left to right: $t=\mynum{0.1}{sh}$;
    $t=\mynum{2.0}{sh}$; $t=\mynum{3.7}{sh}$}\label{fig0p25_temp}
\end{figure}

The simulations using \JLGcode are done on three meshes composed of
nonuniform triangular Delaunay meshes (Mesh4: 115,079 grid points, Mesh5: 330,735 grid points, 
Mesh6: 1,261,299 grid points). The results
are compared to the simulations done with \ETcode using three
quadrangular meshes (Mesh 1: 786,945, Mesh 2: 3,146,753, Mesh 3:
12,584,961). All the meshes (triagular and quadrangular) are more refined in the shell and the interior
region than in the exterior region. We show in Figure~\ref{fig0p25_density} and
Figure~\ref{fig0p25_temp} the scatter plots of the density and
temperature as functions of the radius $r\eqq\|\bx\|_{\ell^2}$ for the
seven meshes in the range $r\in[0,0.16]$ and for times
$\mynum{0.1}{sh}$, $\mynum{2}{sh}$, and $\mynum{3.7}{sh}$.
We observe that cylindrical symmetry is well preserved for all the fields. 

\begin{figure}[ht]\centering
  \includegraphics[width=0.35\linewidth,trim=0 0 0 0,clip=]{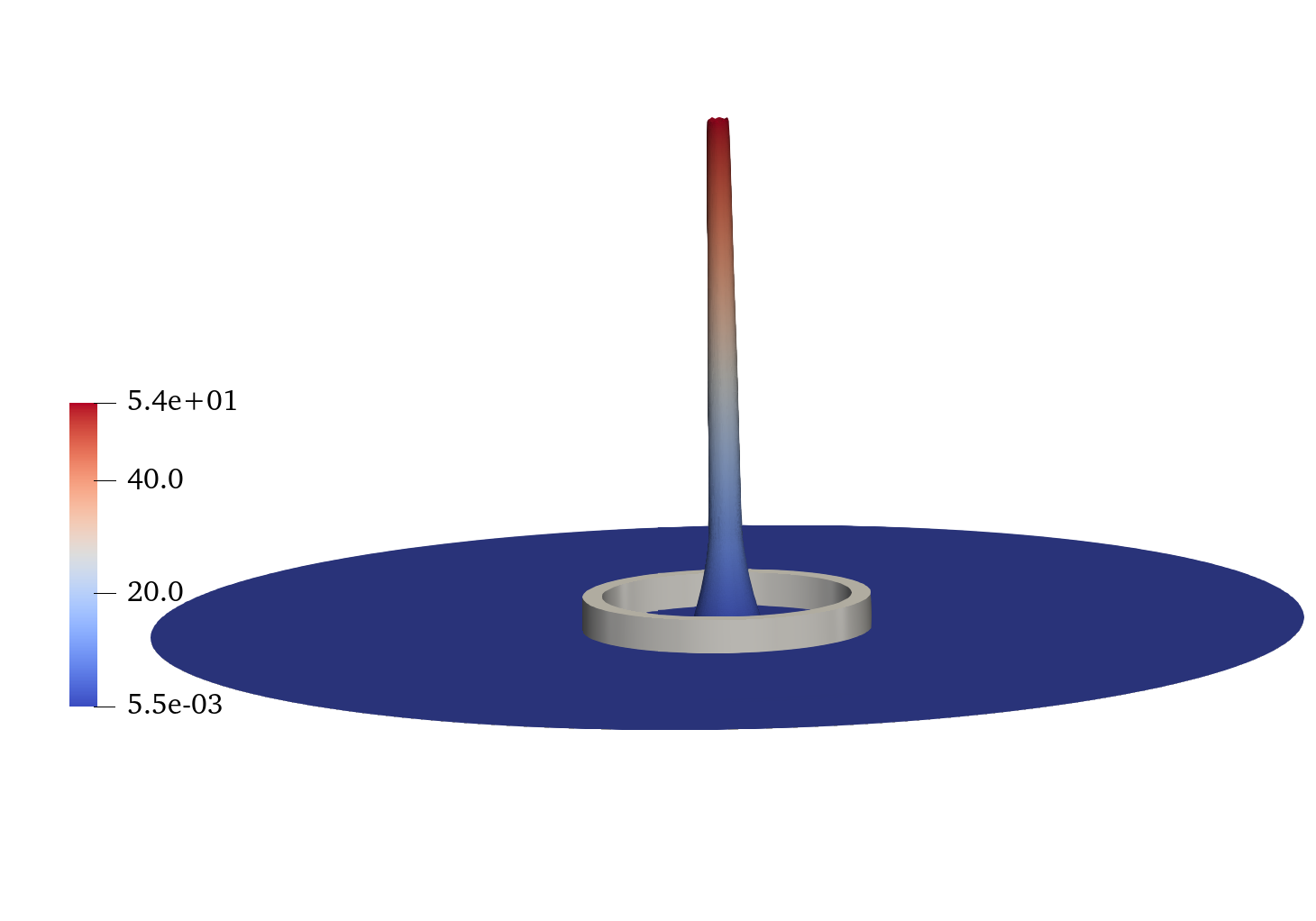}%
  \includegraphics[width=0.29\linewidth,trim=0 0 0 0,clip=]{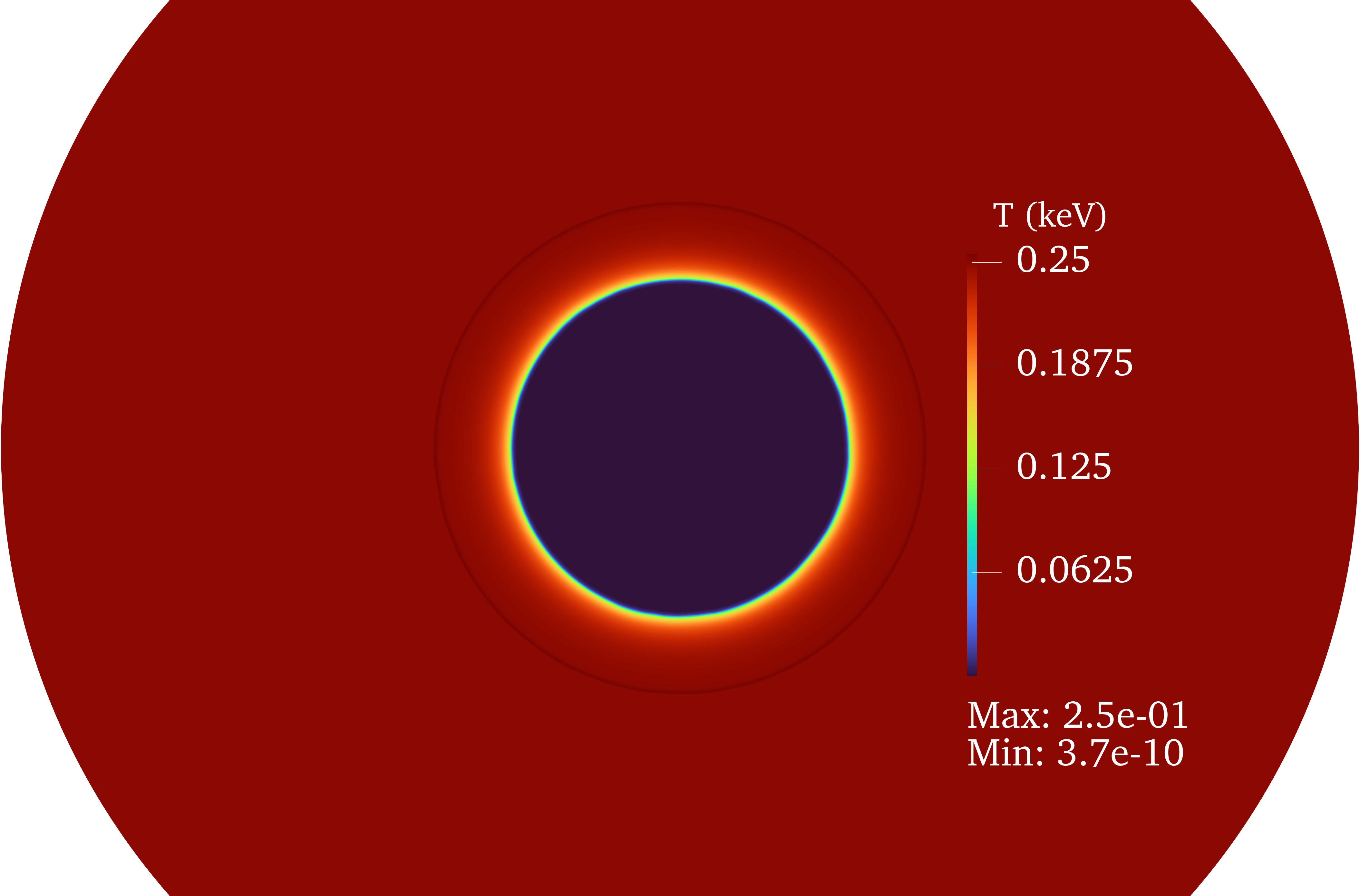}%
   \includegraphics[width=0.35\linewidth,trim=25 11 15 10,clip=]{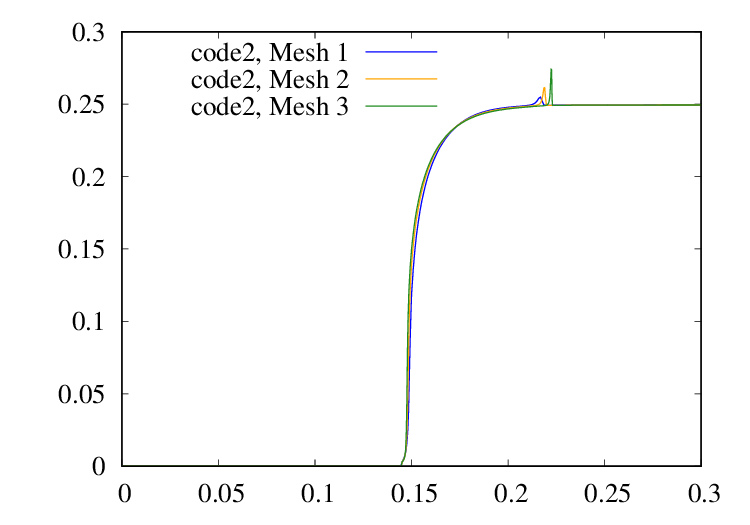}
  \caption{Left: 3D rendering of the density density field; the
    density at $t=0$ is shown is solid color; the density at $t=\mynum{3.7}{sh}$
    is shown in color (\JLGcode, Mesh6). Right: temperature at $t=0.1$; notice the faint circular black line which is 
    a Zeldovich zpike.}
    \label{ICF_2d}
\end{figure}

We further show in the left panel of Figure~\ref{ICF_2d} a
three-dimensional rendering of the density field at
$t=\mynum{0.1}{sh}$ (solid color) and and $t=\mynum{3.7}{sh}$ (using
the cool to warm color palette).  We show the temperature field at $t =
\mynum{0.1}{sh}$ in the center panel. We observe in the temperature
profile the Zeldovich spike propagating outwards away from the center
(black line in the red background). We highlight the Zeldovich spike in the right panel in
Figure~\ref{ICF_2d} by plotting again the temperature
profile over the radius range $r\in[0,0.3]$.

\appendix

\section{First hyperbolic problem and Riemann problem}

We focus in this section on the hyperbolic
problem~\eqref{hyperbolic_one}. We recall here the technique that is
introduced in \cite{Clayton_Guermond_Popov_SIAM_SISC_2022} to
construct an invariant-domain preserving approximation of the problem
for any equation of state satisfying the generic
assumptions~\eqref{Pressure_is_nonnegative} with the domain $\calB(b)$
defined in \eqref{def_of_calB}. The key is to construct auxiliary
states with the desired properties.

\subsection{Abstract Riemann problem and bar states}

Since the technique we are going to present in this section is quite
general, we change notation for a moment and assume that one wants to
solve a general hyperbolic system with some generic flux
$\polg:\calG\to \polR^{s\CROSS d}$, $s\ge 1$, where $\calG$, the
domain of $\polg$, is a subset of $\Real^s$.  Given two states
$\bu_L,\bu_R$ in $\calG$, a unit vector $\bn$ in $\Real^d$, and a
positive real number $\lambda$, we consider the following auxiliary
states (also called bar states hereafter):
\begin{equation} \label{def_ubar}
  \overline\bu_{LR}(\lambda)\eqq \frac{1}{2} (\bu_L+\bu_R) -
  \frac{1}{2\lambda} (\polg(\bu_R)\bn-\polg(\bu_L)\bn).
\end{equation}
To be able to extract information
regarding $\overline\bu_{LR}(\lambda)$, it is useful
to consider the following Riemann problem:
\begin{equation}
  \partial_t \bw + \partial_x(\polg(\bw)\bn) = 0,\qquad
  \bw(x,0)=\begin{cases}
    \bu_L & \text{if $x<  0$,}  \\
    \bu_R & \text{if $0\le x$}.
  \end{cases}\label{Riemann_problem}
\end{equation}
Recall that~\eqref{Riemann_problem} may have infinitely many weak
solutions and weak solutions to \eqref{Riemann_problem} are
self-similar. For every self-similar weak solution, $\bv$, there
exists a number $\lambda_{\max}^\bv>0$, called maximum wave speed, so
that $\bv(\frac{x}{t})= \bu_L$ if $x\le -\lambda_{\max}^\bv t$ and
$\bv(\frac{x}{t})= \bu_R$ if $\lambda_{\max}^\bv t\le x$.  The
following result, proved in Lemma~2.1 and Lemma~2.2 in
\citep{guermond_popov_sinum_2016} (see also Lemma~3.2 in
\cite{Clayton_Guermond_Popov_SIAM_SISC_2022}) and largely inspired
from the work of P. Lax, A. Harten, E. Tadmor \etal (see
Remark~\ref{Rem:literature_on_barstates}), explains the connection
between ~\eqref{def_ubar} and~\eqref{Riemann_problem}.
\begin{lemma}\label{Lem:Riemann_average} Let $\bu_L,\bu_R$ be two
  arbitrary states in $\calG$. Let $\bv(\tfrac{x}{t})$ be any
  self-similar weak solution to~\eqref{Riemann_problem}. Let
  $\lambda_{\max}^\bv$ be the maximum wave speed for this weak
  solution. Let us set
  $\overline\bv(t) :=\int_{-\frac12}^{\frac12} \bv(\tfrac{x}{t})
    \diff x$ for all $t>0$.  Let $\lambda>0$ and assume that
  $\lambda\ge\lambda_{\max}^\bv$. Then
  \begin{align}
     & \overline\bu_{LR}(\lambda) =\overline\bv(\tfrac{1}{2\lambda}),\label{Riemann_avg:Lem:Riemann_average} \\
     & \text{If $\calG$ is convex and $\bv(\xi)\in \calG$ for all $\xi\in \Real$, then
    $\overline\bu_{LR}(\lambda)\in \calG$,}\label{IDP:Lem:Riemann_average}                                   \\
    & \label{Entropy_ineq:Lem:Riemann_average}
    \text{Let $(\eta,\bq)$ be an
      entropy pair.
     If $\bv$ is \sth $\partial_t \eta(\bv) + \partial_x(\bq(\bv)\bn) \le 0$,
    then}                                                     \\
     & \qquad    \eta(\overline\bu_{LR}(\lambda))\le \frac{1}{2} (\eta(\bu_L)+\eta(\bu_R))-
    \frac{1}{2\lambda} (\bq(\bu_R)\bn-\bq(\bu_L)\bn). \nonumber
  \end{align}
\end{lemma}
Let $\bv(\tfrac{x}{t})$ be any self-similar weak solution
to~\eqref{Riemann_problem}, then
\eqref{Riemann_avg:Lem:Riemann_average} says that
$\overline\bu_{LR}(\lambda)$ is equal to the average of $\bv$ over the
interval $(-\frac12,\frac12)$ if $\lambda\ge\lambda_{\max}^\bv$.  As a
result, the statement in~\eqref{IDP:Lem:Riemann_average} says that
under the assumptions that $\calG$ is convex and the weak solution
$\bv$ leaves $\calG$ invariant, then
$\overline\bu_{LR}(\lambda)\in\calG$ as
well. Finally,~\eqref{Entropy_ineq:Lem:Riemann_average} says that if
the weak solution $\bv$ satisfies an entropy inequality for some pair
$(\eta,\bq)$, then $\overline\bu_{LR}(\lambda)$ satisfies a discrete
counterpart of this inequality.

Unfortunately, the Riemann problem~\eqref{Riemann_problem}
cannot be solved analytically in general. But, a key observation made in
\citep{Clayton_Guermond_Popov_SIAM_SISC_2022} that allows us to go around this roadblock is that it is not necessary to solve~\eqref{Riemann_problem} to extract useful
information on the bar states~\eqref{def_ubar}. One can instead
consider a surrogate Riemann problem that is solvable and somewhat
interpolates $\polg$ as we now explain. We take inspiration from~\citep{Clayton_Guermond_Popov_SIAM_SISC_2022} and introduce an
extension technique that will facilitate this interpolation
process.

We assume that we have at hand a new integer $\ws\ge s$, an extension operator
$\Theta: \Real^s \to \Real^{\ws}$, an extended flux
$\wpolg:\wcalD\to \Real^{\ws\CROSS d}$, and  a linear reduction
operator $\Pi: \Real^{\ws}\to \Real^s$, so that the following identities
hold true:
\begin{subequations}\label{extension_reduction_operators}
  \begin{alignat}{2}
     & \Pi(\Theta(\bw)) =\bw,                              &  & \forall\bw\in\calG, \label{extension_reduction_operators:Pi_Theta} \\
     & \Pi(\wpolg(\Theta(\bw))\bn) =  \polg(\bw)\bn, \quad &  & \forall
    \bw\in\calG, \ \forall \bn. \label{extension_reduction_operators:flux}
  \end{alignat}
\end{subequations}
Notice that the triple $(\ws=s,\Pi=\Id,\Theta=\Id)$, where $\Id$ is
the identity operator, trivially satisfies the above assumption,
thereby showing that the class of objects we are considering is not
empty.  Note that
\eqref{extension_reduction_operators:Pi_Theta} implies that $\Theta$
is a right inverse of $\Pi$.  Finally we consider the
following extended Riemann problem:
\begin{equation}
  \partial_t \wbw + \partial_x(\wpolg(\wbw)\bn) = 0,\qquad
  \wbw(x,0)=\begin{cases}
    \Theta(\bu_L) & \text{if $x<  0$,}  \\
    \Theta(\bu_R) & \text{if $0\le x$}.
  \end{cases}\label{extended_Riemann_problem}
\end{equation}
For future reference we also introduce the bar state associated with~\eqref{extended_Riemann_problem}:
\begin{equation}
  \overline\wbu_{LR}(\lambda)\eqq \frac{1}{2} (\Theta(\bu_L)+\Theta(\bu_R)) -
  \frac{1}{2\lambda}
  (\wpolg(\Theta(\bu_R))\bn-\wpolg(\Theta(\bu_L))\bn).
  \label{def_bar_hat_uLR}
\end{equation}
Of course the above somewhat obscure construction has the potential
to be useful only if solving~\eqref{extended_Riemann_problem} is
significantly easier than solving~\eqref{Riemann_problem}.  We show
below that it is indeed the case for the
problem~\eqref{hyperbolic_one}.  The following result is essential
and shows how the bar states $\overline\wbu_{LR}(\lambda)$ and
$\overline\bu_{LR}(\lambda)$ are related.
\begin{lemma}[Extended bar state]\label{Lem:extended_bar_state} Assume
  that the assumptions~\eqref{extension_reduction_operators} are met.
  Then the following identity holds true for all pairs $(\bu_L,\bu_R)\in
    \calG^2$ and all $\lambda>0$:
  \begin{equation}
    \Pi(\overline\wbu_{LR}(\lambda)) =
    \overline\bu_{LR}(\lambda). \label{eq:Lem:extended_bar_state}
  \end{equation}
\end{lemma}
\begin{proof}
  Apply the operator $\Pi$ on both sides of the definition~\eqref{def_bar_hat_uLR}, use the
  linearity of $\Pi$, and conclude using~\eqref{extension_reduction_operators}.
\end{proof}

Hence, to establish that $\overline\bu_{LR}(\lambda)\in\calG$, it
suffices to find a pair of operators
$(\Pi,\Theta)$ and an extended flux $\wpolg$ for which
the extended Riemann problem \eqref{extended_Riemann_problem} can be easily solved, and such that
$\Pi(\overline\wbu_{LR}(\lambda)) \in\calG$.

\begin{remark}[Literature] \label{Rem:literature_on_barstates} The states $\overline{\bu}_{LR}(\lambda)$
  are the backbone of Lax's scheme.  The importance of these states
  has been recognized in \cite[Eq.~(2.6)]{Nessyahu_Tadmor_1990}.  It
  is established in \cite[\S3.A]{Harten_Lax_VanLeer_1983} that these
  states are averages of Riemann solutions provided
  $\lambda\ge \lambda_{\max}$.  The idea of extending the Riemann
  problem to simplify the estimation of $\lambda_{\max}$
  (see~\eqref{extension_reduction_operators}) has its origins in
  \cite[\S3]{Clayton_Guermond_Popov_SIAM_SISC_2022} where this
  construction is used to estimate a guaranteed upper bound on the
  maximum wave speed in the Riemann problem associated with the
  compressible Euler equation supplemented with an arbitrary equation
  of state. The above abstract construction with the operators $\Pi$,
  $\Theta$ and the extended flux $\wpolg$ generalizes
  \citep[\S3]{Clayton_Guermond_Popov_SIAM_SISC_2022}.
\end{remark}

\subsection{Extended flux and Riemann problem}

We present in this section one possible extension of the Riemann
problem~\eqref{Riemann_problem} with $\polg$ defined in
\eqref{hyperbolic_one}.  This is done by proceeding as in
\citep[\S3]{Clayton_Guermond_Popov_SIAM_SISC_2022}. For simplicity we assume that the
pressure oracle is such that the pressure is positive and the cold curve is zero.
We refer to \cite{Clayton_Tovar_2025} and \citep{Guermond_Popov_Saavedra_Sheridan_CMAME_2025}
for generalizations removing these restrictions but still using the above theoretical setting.

We set
$s\eqq d+3$ and $\ws\eqq s+1$. We
define the operators $\Pi$ and $\Theta$ as follows:
\begin{subequations}
  \begin{align}
    \Pi: \Real^{\ws} \ni\wbu \eqq (\rho,\bbm\tr,\Emech,\ER,\Gamma)\tr & \mapsto
    \Pi(\wbu)\eqq (\rho,\bbm\tr,\Emech,\ER)\tr\in \Real^s,                                                                     \\
    \Theta: \Real^s \ni \bu \eqq (\rho,\bbm\tr,\Emech,\ER)\tr         & \mapsto
    \Theta(\bu) \eqq (\rho,\bbm\tr,\Emech,\ER,\Gamma(\bu))\tr\in\Real^{\ws}.                                                   \\
                                                                      & \qquad  \text{with}\quad \Gamma(\bu)\eqq \rho + p(\bu)
    \frac{1-b\rho}{e(\bu)}.\nonumber
  \end{align}
\end{subequations}
Then, given $\wbu\eqq (\rho,\bbm\tr,\Emech,\ER,\Gamma)\tr$, we define the
extended flux
\begin{align}
  \wpolg(\wbu)\eqq \begin{pmatrix}
                     \bv\rho \\ \bv\otimes\bbm + \wp(\wbu)\polI_d \\ \bv(E\lo{m}+\wp(\wbu)) \\
                     \bv \ER \\
                     \bv\Gamma\end{pmatrix},\quad \text{with}\quad \wp(\wbu)\eqq
  \frac{(\Gamma-\rho)}{1-b\rho}e(\Pi(\wbu)).
  \label{def_flux_radiation_hyp}
\end{align}

\begin{lemma} The operators $\Pi$, $\Theta$, and $\wpolg$ satisfy the assumptions
  \eqref{extension_reduction_operators}.
\end{lemma}
\begin{proof}
  Let us verify that
  \eqref{extension_reduction_operators:Pi_Theta} holds. Let
  $\bu\eqq(\rho,\bbm,\Emech,\ER)\tr\in\Real^s$.
  Then
  \begin{align*}
    \Pi(\Theta(\bu))= \Pi((\rho,\bbm,\Emech,\ER,\Gamma(\bu))\tr)=(\rho,\bbm,\Emech,\ER)\tr=\bu.
  \end{align*}
  Let us now verify that~\eqref{extension_reduction_operators:flux}
  holds. Let $\bu\eqq(\rho,\bbm,\Emech,\ER)\tr$ be an arbitrary state in
  $\calB(b)$. Then we observe that
  \begin{align*}
    \wp(\Theta(\bu))
    =\frac{(\Gamma(\bu)-\rho)}{1-b\rho}e(\Pi(\Theta(\bu)))
    = \frac{p(\bu)}{e(\bu)} e(\bu)= p(\bu).
  \end{align*}
  This gives
  $\wpolg(\Theta(\bu))\eqq\left(
    \bv\tr\rho, \bv\otimes\bbm + p(\bu)\polI_d, \bv\tr(\Emech+p(\bu)),
    \bv\tr \ER,\bv\tr\Gamma(\bu)\right)\tr$,
  and the identity~\eqref{extension_reduction_operators:flux} readily follows.
\end{proof}

\begin{lemma}[Existence\&uniqueness]
  For all pairs of states $(\bu_L,\bu_R)$ in $\calA(b)$, the extended
  Riemann problem~\eqref{extended_Riemann_problem} has a unique
  self-similar solution $\bv$ that is entropic in the sense of Lax (see
  \citep[\S7]{Lax_1957_II}), and $\calA(b)$ in invariant for $\Pi(\bv)$
\end{lemma}
\begin{proof}
  The construction of the solution is essentially the same as \S4 in
  \citep{Clayton_Guermond_Popov_SIAM_SISC_2022} with the exception
  that now the extended state contains the radiation energy.

  Let us start by verifying that the states
  $\Theta(\bu_L),\Theta(\bu_R)$ are admissible to be able to use this
  construction. Let $Z$ be in the index set $\{L,R\}$.  By definition
  we have
  $\wbu_Z\eqq\Theta(\bu_Z)\eqq(\rho_Z,\bbm_Z, (\Emech)_Z , (\ER)_Z,
    \Gamma(\bu_Z))\tr$.  As $\bu_Z\in\calA(b)$, we have $e(\bu_Z)>0$,
  and owing to the assumption~\eqref{Pressure_is_nonnegative}, the
  tuple $(\rho_Z,\bbm_Z, (\Emech)_Z, \Gamma(\bu_Z))\tr$ is such
  that $\Gamma_Z\eqq \Gamma(\bu_Z)>1$. Notice that the radiation
  energy is a passive scalar (\ie
  $\partial_t\ER + (\bu\SCAL\bn)\partial_\bx \ER =0$); hence, the
  radiation energy stays constant on each side of the contact wave,
  and it is therefore necessarily positive since $\bu_Z\in\calA(b)$.
  As $\ER$ is just a passive scalar in the Riemann problem and it is
  not coupled with the other components of the Riemann solution, we
  can apply the theory explained in \S4 in
  \citep{Clayton_Guermond_Popov_SIAM_SISC_2022} to construct a unique
  self-similar solution $\bv$ that is entropic in the sense of
  Lax. This solution satisfies $\Pi(\bv)\in \calB(b)$.
\end{proof}

\begin{corollary}[Bar states]
  For all $(\bu_L,\bu_R)$ in $\calA(b)$ and all unit vector $\bn$ in $\Real^d$, let
  $\wlambda_{\max}(\bn,\bu_L,\bu_R)$ be any upper bound on the maximum
  wave speed in the extended Riemann problem
  ~\eqref{extended_Riemann_problem}.  Let
  $\overline\bu_{LR}(\lambda)$ be  the bar state defined in
  ~\eqref{def_ubar}.
  Then $\overline\bu_{LR}(\lambda)\in\calA(b)$
  for all $\lambda\ge \wlambda_{\max}(\bn,\bu_L,\bu_R)$.
\end{corollary}

Note that the definition of the extended pressure in
\eqref{def_flux_radiation_hyp} makes the extended Riemann
problem~\eqref{extended_Riemann_problem} easy to solve for any
pressure oracle satisfying \eqref{Pressure_is_nonnegative}.  The
oracle is only invoked to compute the two pressures $p_L$ an $p_R$.
On the left of the contact wave, the ratio
$\gamma\eqq \frac{\Gamma}{\rho}$ is constant and  equal to
$\gamma_L\eqq 1 + p_L\frac{1-b\rho_L}{\rho_Le_L}$, the extended
pressure is equal to $(\gamma_L-1)\frac{\rho e(\Pi\wbu)}{1-b\rho}$.
On the right of the contact wave, the ratio $\gamma\eqq
  \frac{\Gamma}{\rho}$ is also constant
and equal to $\gamma_R\eqq 1  + p_R\frac{1-b\rho_R}{\rho_Re_R}$, the extended
pressure is equal to $(\gamma_R-1)\frac{\rho e(\Pi\wbu)}{1-b\rho}$.
A source code providing the upper bound
$\wlambda_{\max}(\bn,\bu_L,\bu_R)$ is publicly available at
\cite{guermond_jean_luc_2021_4685868}.

\section{Second Riemann problem} \label{Sec:Second_Riemann_Problem}
We study the Riemann problem associated with the hyperbolic system
\eqref{hyperbolic_two} and derive an upper bound on the maximum
wave speed for this Riemann problem. The main result of this section
is Lemma~\ref{Lem:wave_speed_upper_bound}.

\subsection{Formulation of the problem}
Let $\bn$ be a unit vector in
$\Real^d$.
Let $\bsfu_L\eqq (\varrho_L,\bsfm_L\tr,\sfE_{\sfm,L},\sfE_{\sfr,L})\tr$
and $\bsfu_R\eqq (\varrho_R,\bsfm_R\tr,\sfE_{\sfm,R},\sfE_{\sfr,R})\tr$ be given
left and right states. Let us set $v_L\eqq \bsfv_L\SCAL\bn$,
and $E_L\eqq \sfE_{\sfr,L} +
  \frac12\rho_L\sfv_L^2$.
Similarly we define $v_R\eqq \bsfv_R\SCAL\bn$,
and $E_R\eqq \sfE_{\sfr,R} +
  \frac12\rho_R\sfv_R^2$.
Then, the Riemann problem associated with the hyperbolic system
\eqref{hyperbolic_two} reduces to solving
\begin{subequations}\label{app_hyperbolic_two}
  \begin{align}
     & \partial_t\rho = 0                                            \\
     & \partial_t (\rho v) +\partial_x p_\sfr=0,\qquad  p_{\sfr}\eqq
    \tfrac13( E-\tfrac{\rho}{2}v^2),                                 \\
     & \partial_t E + \partial_x (v p_\sfr) =0,
  \end{align}\end{subequations}
with left and right states $(\varrho_L,\sfv_L,\sfE_L)\tr$ and
$(\varrho_R,\sfv_R,\sfE_R)\tr$, respectively.  Using the change of
variables
$(\bv,E)\tr \mapsto (v,E_\sfr \eqq E-\tfrac{\rho}{2}v^2)\tr$, the
nontrivial part of the above system can be rewritten in the following form
for which the computation of the eigenvalues of the Jacobian matrix of the flux is easier
\[
  \partial_t \begin{pmatrix}v\\ E_\sfr\end{pmatrix} =
  -\frac13 \begin{pmatrix}\frac{1}{\rho} \partial_x E_\sfr \\
    E_\sfr\partial_x v\end{pmatrix} = -\frac13 \begin{pmatrix} 0 &
                \frac{1}{\rho} \\ E_\sfr & 0\end{pmatrix} \partial_x \begin{pmatrix}v\\ E_\sfr\end{pmatrix}.
\]
The eigenvalues of the Jacobian matrix
are $\pm\frac13 \sqrt{E_\sfr/\rho}$.  Hence, the system
\eqref{app_hyperbolic_two} is hyperbolic; the three eigenvalues are
$-\frac13\sqrt{E_\sfr/\rho}$, $0$, and $\frac13\sqrt{E_\sfr/\rho}$.
The eigenvalues $-\frac13\sqrt{E_\sfr/\rho}$ and
$\frac13\sqrt{E_\sfr/\rho}$ are genuinely nonlinear. The eigenvalue
$0$ is linearly degenerate and is associated with a contact wave.

\subsection{Rarefaction wave}

We first construct the rarefaction solution associated with the left
wave. Recall that the density is equal to $\rho_L$ in the left wave.
Let $\xi\eqq x/t$ be the self-similarity variable. Let us abuse the
notation and let us denote $v(\xi)$ the velocity and $p(\xi)$ the
pressure. The momentum conservation equation reduces to
$-\xi \partial_\xi v + \rho_L^{-1}\partial_\xi p=0$. Recalling that
$E_\sfr = 3 p$, we have $\xi = -\sqrt{p/3\rho_L}$ in the left wave.
The conservation equation is
\begin{equation}
  \sqrt{\frac{\rho_L}{12}}\partial_\xi v + \frac{1}{2\sqrt{p}}\partial_\xi p =0.
\end{equation}
This implies that the dependency with respect to the pressure of the
velocity and the wave speed in the left rarefaction wave is given by
\begin{equation}
  v = -\sqrt{\frac{12}{\rho_L}}\left(\sqrt{p} -\sqrt{p_L} \right) + v_L,
  \qquad \lambda_L(p) = -\sqrt{\frac{p}{3\rho_L}}.
\end{equation}
Using the same argument, and recalling that the self-similar variable
is $\xi = \sqrt{p/3\rho_R}$ in the right wave, we obtain that the
dependency with respect to the pressure of the velocity and the
wave speed in the right rarefaction wave is given by
\begin{equation}
  v = \sqrt{\frac{12}{\rho_R}}\left(\sqrt{p}- \sqrt{p_R}\right) + v_R,
  \qquad \lambda_R(p) =\sqrt{\frac{p}{3\rho_R}}.
\end{equation}

\subsection{Shock wave}

If the left wave is a shock, the Rankine-Hugoniot relation
implies that there exists $s<0$ so that
\begin{equation}
  s\rho_L(v-v_L)=p-p_L,\qquad s(3 p+\tfrac12 \rho_L v^2-3 p_L-\tfrac12
  \rho_Lv_L^2)= vp-v_Lp_L.
  \label{RH_left_shockwave}
\end{equation}
where used that $E=3 p+\frac12\rho_L v^2$. Hence
\begin{align*}
  (3 p+\tfrac12 \rho_L v^2-3 p_L-\tfrac12 \rho_Lv_L^2)(p-p_L)    & = \rho_L( vp-v_Lp_L) (v-v_L) \\
  & \implies \\ 
  - \tfrac12 \rho_L (p+p_L) v^2 +\rho_Lv_L(p_L+p) v + 3(p-p_L)^2 & -\tfrac12 \rho_Lv_L^2(p+p_L)=0.
\end{align*}
The discriminant of the quadratic equation in $v$ is
\[
  \Delta \eqq
  6(p-p_L)^2 \rho_L (p+p_L).
\]
Recalling that $s<0$ and the pressure increases  along the
shock curve (\ie $p\ge p_L)$, we conclude from \eqref{RH_left_shockwave}  that the
solution to the quadratic equation must be such that $v-v_L\le 0$;
hence, the velocity and the wave speed in left shock solution are given by
\[
  v
  = v_L - \sqrt{\frac{6}{\rho_L}} \frac{(p-p_L)}{\sqrt{p+p_L}},
  \qquad
  \lambda_L(p) = -\sqrt{\frac{p+p_L}{6\rho_L}}.
\]
Similarly, the shock solution in the right wave is given by
\[
  v
  = v_R + \sqrt{\frac{6}{\rho_R}} \frac{(p-p_R)}{\sqrt{p+p_R}},
  \qquad \lambda_R(p) =\sqrt{\frac{p+p_R}{6\rho_R}}.
\]

\subsection{Description of the solution}

We define the index set $\{L,R\}$, and for all $Z$ in the index set $\{L,R\}$ we define
the function $f_Z:\Real_{\ge 0}\to\Real$
\begin{equation}
  f_Z(p) =
  \begin{cases}
    \sqrt{\frac{6}{\rho_Z}} \frac{(p-p_Z)}{\sqrt{p+p_Z}}      & \text{if
    $p_Z\le p$ (shock)},                                                 \\
    \sqrt{\frac{12}{\rho_Z}}\left(\sqrt{p} -\sqrt{p_Z}\right) & \text{if
      $p\le p_Z$ (expansion).}
  \end{cases}
\end{equation}
This definition implies that the velocity in the left wave is given by
$v=-f_L(p)+v_L$ and the velocity in the right wave is given by
$v=f_R(p)+v_R$. The left and right waves can be continuously reconnected only if there exists
$p^*\ge 0$ so that $f_R(p^*)+v_R=-f_L(p^*) + v_L$. We thus define
$\phi(p) = f_R(p)+f_L(p) +v_R- v_L$ for all $p\ge 0$.
The function $\phi$ is monotone strictly increasing.
Hence, the pressure $p^*$ connecting the left and right waves solves the nonlinear equation $\phi(p^*)=0$.
Once $p^*$ is found, the extreme wave speed of the left and right
waves are
\begin{equation}
  \lambda_L^- = -\sqrt{\frac{p_L+ \max(p^*,p_L)}{6\rho_L}},
  \qquad \lambda_R^+ =\sqrt{\frac{p_R+\max(p^*,p_R)}{6\rho_R}}.\label{def_lambda_L_R}
\end{equation}
We finish this section by showing how $p^*$ can be estimated from above.

\subsubsection{Case 0: Vacuum, $p^*=0$} If $\phi(0)>0$, then the equation
$\phi(p) =0$ has no root. This means that vacuum forms between the
left and the right waves.  Vacuum forms when
\begin{equation}
  v_R-v_L
  -\sqrt{\frac{12\, p_L}{\rho_L}}-\sqrt{\frac{12\, p_R}{\rho_R}}>0.
\end{equation}
In this case we conventionally set $p^*=0$. Note that if $\phi(0)=0$,
then $p^*=0$ is the unique solution. The left and right waves are both
expansions.

\subsubsection{Case 1: $0<p^*$ and $0<\phi(p_{\min})$}
Let us denote $p_{\min}\eqq \min(p_L,p_R)$.
The condition
$0<\phi(p_{\min})$ implies that the left and the right waves are
both expansions (since $p^{*} < p_{\min}$), and we have
\begin{equation}
  p^* = \left(\frac{v_L-v_R +\sqrt{\frac{12\, p_L}{\rho_L}}
    +\sqrt{\frac{12\, p_R}{\rho_R}}}{\sqrt{\frac{12}{\rho_L}}+\sqrt{\frac{12}{\rho_R}}}\right)^2.
\end{equation}

\subsubsection{Case 2: $\phi(p_{\min})< 0<\phi(p_{\max})$} Let us
denote  $p_{\min}\eqq \min(p_L,p_R)$ and  $p_{\max}\eqq \max(p_L,p_R)$. The
solution is composed of an expansion and a shock when $\phi(p_{\min})<
  0<\phi(p_{\max})$.
The root of $\phi(p) = 0$ can be computed by using verbatim Algorithm~2
from~\citep{Guermond_Popov_2016_JCP}. The algorithm can be initialized
by using $p_{\min}$ and $p_{\max}$ as lower and upper bounds on $p^*$,
respectively.

\subsubsection{Case 3: $\phi(p_{\max})<0$}
The solution is composed of two shocks when $\phi(p_{\max})<0$.
That is,
\[
  \phi(p)= \sqrt{\frac{6}{\rho_L}}\frac{p-p_L}{\sqrt{p+p_L}}
  + \sqrt{\frac{6}{\rho_R}}\frac{p-p_R}{\sqrt{p+p_R}} + v_R-v_L.
\]
Note that an upper bound on $p^*$ can be
obtained by computing the zero of the following function:
\[
  \widetilde \phi(p)\eqq \sqrt{\frac{3}{\rho_L}}\frac{p-p_L}{\sqrt{p}}
  + \sqrt{\frac{3}{\rho_R}}\frac{p-p_R}{\sqrt{p}} + v_R-v_L,
\]
since $\phi(p)>\widetilde \phi(p)$ for all $p>p_{\max}$. The unique zero of $\widetilde \phi(p)=0$
is obtained by making the substitution $\sqrt{p}\rightarrow x$ and solving the quadratic equation $a x^2 + b x + c=0$ with
\[
  a\eqq  \sqrt{\frac{3}{\rho_L}}+\sqrt{\frac{3}{\rho_R}},\quad
  b\eqq v_R-v_L,\qquad
  c \eqq -\sqrt{\frac{3}{\rho_L}}p_L-\sqrt{\frac{3}{\rho_R}}p_R.
\]
Letting $x_+$ be the largest root of the quadratic equation, and setting $\widetilde p^*\eqq x_+^2$,
we necessarily have $p^*< \widetilde p^*$.
Then, the zero of $\phi$ can be computed by using verbatim Algorithm~2
from \citep{Guermond_Popov_2016_JCP} and initializing the algorithm
with $p_{\max}$ as lower bound and $\widetilde p^*$ as upper bound.

\subsection{Conclusion}

The main result of this section is the following result.

\begin{lemma}[Wave speed uper bound]\label{Lem:wave_speed_upper_bound}
  With $p^*$ defined above and the definitions~\eqref{def_lambda_L_R},
  a guaranteed upper bound on the maximum wave speed in the Riemann
  problem~\eqref{app_hyperbolic_two}
  is given by
  \begin{equation}
    \wlambda_{\max}(\bn,\bw_L,\bw_R) = \max(-\lambda_L^-(p^*),\lambda_R^+(p^*)).
  \end{equation}
\end{lemma}

\bibliographystyle{abbrvnat}
\bibliography{ref}

\end{document}